\documentclass{article}
\usepackage{amsmath,amsthm,amssymb}
\usepackage{tikz-cd}
\usepackage{todonotes}
\usepackage{hyperref}

\newtheorem{thm}{Theorem}[section]
\newtheorem{prop}[thm]{Proposition}
\newtheorem{lemma}[thm]{Lemma}
\newtheorem{cor}[thm]{Corollary}
\newtheorem{rmk}[thm]{Remark}

\theoremstyle{definition}
\newtheorem{example}[thm]{Example}
\newtheorem{defn}[thm]{Definition}

\newcommand{\im}{\operatorname{im}}

\newcommand{\EE}{\mathbb{E}}
\newcommand{\DD}{\mathbb{D}}
\newcommand{\FF}{\mathbb{F}}
\newcommand{\BB}{\mathbb{B}}
\newcommand{\CC}{\mathbb{C}}
\newcommand{\cart}{\operatorname{Cart}}
\newcommand{\dom}{\operatorname{dom}}
\newcommand{\cod}{\operatorname{cod}}

\newcommand{\xalg}[1]{#1\text{-}\operatorname{Alg}}
\newcommand{\xcoalg}[1]{#1\text{-}\operatorname{Coalg}}
\newcommand{\psh}[1]{\sets^{\mathcal{#1}^{\op}}}
\newcommand{\yoneda}{\mathbf{y}}
\newcommand{\pbcorner}{\ar[dr, phantom, very near start, "\lrcorner"]}

\newcommand{\sets}{\mathbf{Set}}
\newcommand{\cats}{\mathbf{Cat}}
\newcommand{\op}{\operatorname{op}}
\newcommand{\sset}{\sets^{\Delta^{\op}}}
\newcommand{\II}{\mathbb{I}}
\newcommand{\RR}{\mathbb{R}}

\newcommand{\decble}{\operatorname{dec}}

\newcommand{\fmly}{\operatorname{Fam}}

\newcommand{\simpcat}{\Delta}

\title{Definable and Non-definable Notions of Structure}
\author{Andrew W Swan}

\begin{document}

\maketitle

\begin{abstract}
  Definability is a key notion in the theory of Grothendieck
  fibrations that characterises when an external property of objects
  can be accessed from within the internal logic of the base of a
  fibration. In this paper we consider a generalisation of
  definability from properties of objects to structures on objects,
  introduced by Shulman under the name local representability.

  We first develop some general theory and show how to recover
  existing notions due to B\'{e}nabou and Johnstone as special cases.
  We give several examples of definable and non definable notions of
  structure, focusing on algebraic weak factorisation systems, which
  can be naturally viewed as notions of structure on codomain
  fibrations. Regarding definability, we give a sufficient criterion
  for cofibrantly generated awfs's to be definable, generalising a
  construction of the universe for cubical sets, but also including
  some very different looking examples that do not satisfy tininess in
  the internal sense, that exponential functors have a right
  adjoint. Our examples of non definability include the identification
  of logical principles holding for the interval objects in simplicial
  sets and Bezem-Coquand-Huber cubical sets that suffice to show a
  certain definition of Kan fibration is not definable.
\end{abstract}

\section{Introduction}
\label{sec:introduction}

\subsection{Definability}
\label{sec:definability}

In na\"{i}ve category theory one often makes use of an external notion
of set. For example, in locally small categories \(\hom(X,Y)\) is a
set, a complete category is one with all small limits and the general
adjoint functor theorem makes essential use of the solution set
condition. This ties the definitions and results to an often
unspecified theory of sets, typically understood to be Zermelo-Fraenkel
set theory with choice. This use of set theory is often unnecessary
and the link can be severed through the use of Grothendieck
fibrations. We think of the base of the fibration as the foundation of
mathematics where we are working. This could be ``the'' category of
sets via a set indexed family fibration, but could also be a specific
model of \(\mathbf{ZF}\), or more generally an elementary topos, or
even more generally a category satisfying even weaker conditions. For
example, in this paper we will consider examples where the base is a
locally cartesian closed category, and examples where the base is the
category of all small categories.

When working over a fibration, it is useful to know when an external
property of objects in the total category can be referred to from
within the internal logic of the base. This idea can be captured
surprisingly well through an elegant notion due to B\'{e}nabou referred
to as \emph{definability} \cite{benaboufibcat}.

Whereas B\'{e}nabou's definition referred only to \emph{properties} of
objects, the same idea can be applied to \emph{structures} on
objects. For example, on a fibration of vertical maps \(V(\EE) \to
\BB\) we could consider the class of maps with the property of being a
split epimorphism, and given a map, we can consider the collection of
all sections witnessing the map as a split epimorphism. To give a more
extreme example, when working in the internal logic of a topos we can
talk about a given object having the property of admitting a group
structure, but it is more useful to be able to talk about the
collection of group structures on an object.

The concept of definability was generalised to certain structures by
Johnstone \cite[Section B1.3]{theelephant} under the name
\emph{comprehension schemes}. A related idea was also considered early
on by Lawvere \cite{lawverecomp}. However, in this paper we will
consider an alternative definition due to Shulman \cite[Section
3]{shulmaninftytopunivalence}. Although our definition is based on and
essentially equivalent to Shulman's we will give a reformulation that
emphasises its role as a generalisation of the earlier ideas by
B\'{e}nabou and Johnstone. For this reason we will mainly use the
terminology \emph{definable}, following B\'{e}nabou in place of
Shulman's \emph{locally representable}.

\subsection{Algebraic weak factorisation systems}
\label{sec:algebr-weak-fact-1}

The concept of \emph{weak factorisation system} (wfs) is fundamental
in homotopical algebra, as a key ingredient in Quillen's definition of
model category \cite{quillen67}. For any weak factorisation system on
a category \(\CC\), the class of right maps, gives a class of objects
in the codomain fibration \(\cod : \CC^\to \to \CC\) which is closed
under reindexing, including as a special case the fibrations in a
model structure. As observed by Shulman
\cite{shulmaninftytopunivalence}, when working in the semantics of
type theory it is natural to ask when the fibrations in a model
category are definable, or failing that, when they can be replaced by
a definable notion of structure. Given a definable notion of
structure, it is straightforward to construct universes that can be
used when modelling type theory. In the absence of a definable notion
of structure there is not a clear way to define universes for type
theory in general.

\emph{Algebraic weak factorisation systems} (awfs) are a structured
version of wfs, first introduced by Grandis and Tholen under the name
\emph{natural weak factorisation system} \cite{grandistholennwfs}. In
an awfs the class of left maps is replaced by the category of
coalgebras for a comonad, and the class of right maps by a category of
algebras for a monad. In \cite{bourkegarnerawfs1} Bourke and Garner
gave a new alternative definition of awfs, proving that is equivalent to the
earlier definition. According to this definition we can understand
awfs's as monadic notions of structure on a codomain fibration,
together with some extra structure in the form of a ``composition
functor.'' Presented like this, we can naturally define awfs's as
being definable simply when the underlying notion of structure is
definable.

When studying the semantics of homotopy type theory constructively,
e.g. as in
\cite{awodey19,vdbergfaber,bchcubicalsets,coquandcubicaltt,gambinohenry,lops,
  pittsortoncubtopos}, it is usual to define the universe of small
fibrations not as small maps that are Kan fibrations, but as small
maps \emph{together with fibration structure}. For this reason, it is
more natural to consider Kan fibrations as part of an awfs, rather
than as merely a class of maps in a wfs. Just as for wfs's, when
constructing the universe it is natural to ask that the awfs is
definable. Although it is unclear whether definability is strictly
necessary to model universes in type theory,\footnote{We could also
  consider the weaker requirement that given a fibration structure on
  a map \(f\) we can witness \(f\) as a pullback of the universe map
  \(\tilde{U} \to U\) in a not necessarily unique way, and this may be
  sufficient for the semantics of type theory.} it does appear to play
an important role in the constructive models of type theory known to
the author, including all those in the references above.

Aside from the semantics of type theory, the question of definability
of awfs's is an interesting topic for two reasons.

As explained above, definability is a rich topic in itself, and from
this point of view awfs's are a source of interesting examples both of
definability and non-definability. In many of the other examples we
will see in this paper definability is something that we can get ``for
free'' from general arguments, often using local smallness of a
fibration, or has no chance at all of holding. On the other hand,
awfs's provide examples of notions of structure where definability
holds or does not hold for non trivial reasons.
We will show that definability holds
whenever an awfs is cofibrantly generated by a family of maps whose
codomain is a ``family of tiny objects,'' recovering some known
instances of definability as a corollary. We will also see some non
trivial examples of awfs's that are \emph{not} definable.

Secondly, we can view the question of definability as a natural one
within the field of awfs's. Definable awfs's have yet to be studied in
detail, but we can already observe the following interesting
property. One of the key properties of awfs's that improves the
situation with wfs's is that left maps are closed under colimits, and
right maps are closed under limits. More precisely, if we are given a
diagram of right maps in a wfs, then the limit is not necessarily a
right map. However, if we are given a diagram of maps that factors
through the category of right maps in an awfs, then the limit is a
right map, simply as corollary of the fact that the forgetful functor
on right map structures is monadic, and so creates limits. In a
definable awfs right maps are also stable under certain colimits,
namely those for diagrams that factor through the category of right
map structures and cartesian homomorphisms. This can be seen as a
corollary of the fact that for a definable awfs the right maps are
both monadic and comonadic in the following sense: the usual category
of right maps and \emph{all} homomorphisms is monadic, as for any
awfs, whereas definability precisely tells us that the wide
subcategory of cartesian homomorphisms is comonadic.

\subsection*{A note on set theoretic foundations}
\label{sec:note-set-theoretic}

In some places of this paper we made use of categories of presheaves
on large categories. This is unproblematic in the presence of
sufficient large cardinals. However, the aforementioned presheaf
categories are only used together with simple algebraic arguments that
can easily be adapted into direct arguments that do not require the
presheaf categories to exist. In this way the main results of this
paper do not depend on large cardinals, or indeed on much set theory
at all.

To avoid the use of the axiom of choice, we follow the convention that
all categorical structure is ``cloven.'' That is, whenever we require
the existence of a collection of such objects as limits, colimits and
cartesian maps, we in fact require an operator assigning a choice of
these objects.

\subsection*{Acknowledgements}
\label{sec:acknowledgements}

I'm grateful for several helpful discussions, comments and suggestions
on this topic from Mathieu Anel, Carlo Angiuli, Steve Awodey, Benno
van den Berg, Jonas Frey, Mike Shulman and Thomas Streicher.

This material is based upon work supported by the Air Force Office of
Scientific Research under award number FA9550-21-1-0009. Any opinions,
findings, and conclusions or recommendations expressed in this
material are those of the author(s) and do not necessarily reflect the
views of the United States Air Force.

\section{Some background and useful lemmas}
\label{sec:some-backgr-usef}

We start with a few basic observations about discrete fibrations and
adjunctions over a fibration including some key lemmas.

\subsection{Discrete fibrations}
\label{sec:discrete-fibrations}

Suppose we are given a discrete fibration $P : \CC \to \DD$. For each
$D \in \DD$ we define a ``local'' version, $P_D$ as the following
pullback.
\begin{equation}
  \label{eq:6}
  \begin{tikzcd}
    (P \downarrow D) \ar[r] \ar[d, swap, "P_D"] \pbcorner & \CC \ar[d, "P"] \\
    \DD/D \ar[r, "\dom"] & \DD
  \end{tikzcd}
\end{equation}

\begin{defn}
  Given any functor $F : \CC \to \DD$ and an object $D \in \DD$ we say
  the \emph{right adjoint to $F$ is defined at $D$} if we are given a
  terminal object of the comma category $(P \downarrow D)$.
\end{defn}

\begin{lemma}
  \label{lem:repriffpartialra}
  Let $P : \CC \to \DD$ be a discrete fibration and $D$ an object of
  $\DD$. The right adjoint to $P$ is defined at $D$ if and only if the
  discrete fibration $P_D$ defined in \eqref{eq:6} is representable as
  a presheaf.
\end{lemma}

\begin{proof}
  We recall that a discrete fibration corresponds to a representable
  presheaf if and only if its domain has a terminal object. However,
  in \eqref{eq:6} we explicitly described the domain of $P_D$ as the
  comma category $(P \downarrow D)$.
\end{proof}

One of the key ideas in our presentation of the general theory of
definability will be the link between representability, existence of a
right adjoint, and comonadicity, which will be a special case of the
lemma below.
\begin{lemma}
  \label{lem:discfibra}
  The following are equivalent.
  \begin{enumerate}
  \item \label{item:1}
    For every $D \in \DD$, the discrete fibration $P_D$ defined in
    \eqref{eq:6} is representable as a presheaf.
  \item $P$ has a right adjoint.
  \item $P$ is comonadic.
  \end{enumerate}
\end{lemma}

\begin{proof}
  For $1 \Leftrightarrow 2$ we apply Lemma \ref{lem:repriffpartialra},
  recalling that $P$ has a right adjoint if and only if the right
  adjoint is defined at $D$ for all objects $D$ of $\DD$.  
  
  To see $2 \Rightarrow 3$, note that discrete fibrations create
  connected limits,\footnote{In general a Grothendieck fibration
    creates any limits that exist in its fibres and are preserved by
    reindexing.}
  and in particular create $P$-split equalizers and
  so we can apply Beck's theorem to show $P$ is comonadic.
\end{proof}

Although we are interested in discrete fibrations that have a
right adjoint, we note in passing that discrete fibrations
do not have left adjoints except in the trivial case.

\begin{prop}
  A discrete fibration $P : \CC \to \DD$ has a left adjoint iff it is
  an isomorphism of categories.
\end{prop}

\begin{proof}
  Let $D \in \DD$. Note that we can explicitly describe the comma
  category $(D \downarrow P)$ as follows. Each object $D \to P C$
  corresponds to an object $\tilde{D}$ such that $P \tilde{D} = D$ and
  a map $\tilde{D} \to C$ in $\CC$. A morphism is a map $C \to C'$
  making a commutative triangle as below
  \begin{displaymath}
    \begin{tikzcd}
      D \ar[r] \ar[dr] & P C \ar[d] \\
      & P C'
    \end{tikzcd}
  \end{displaymath}
  If such a morphism exists the two objects must have the same
  $\tilde{D}$.

  Hence if $(D \downarrow P)$ has an initial object then $P^{-1}(D)$
  contains exactly one object. It follows that $P$ is full, and $P$ is
  faithful in any case. Hence $P$ is an isomorphism of categories.
\end{proof}

\subsection{Adjunctions over fibrations}
\label{sec:adjunct-over-fibr}

\begin{defn}
  An adjunction \emph{over} \(\BB\) consists of fibrations \(p : \DD \to \BB\)
  and \(q : \EE \to \BB\) together with an adjunction \(F \dashv G :
  \DD \to \BB\) such that \(F\) and \(G\) commute with \(p\) and
  \(q\), as illustrated below, and the unit and counit of the
  adjunction can be chosen to be vertical.
  \begin{displaymath}
    \begin{gathered}
      \begin{tikzcd}
        \DD \ar[rr, "F"] \ar[dr, swap, "p"] & & \EE \ar[dl, "q"] \\
        & \BB &
      \end{tikzcd}
    \end{gathered}
    \qquad
    \begin{gathered}
      \begin{tikzcd}
        \DD \ar[dr, swap, "p"] & & \EE \ar[dl, "q"] \ar[ll, swap, "G"] \\
        & \BB &        
      \end{tikzcd}
    \end{gathered}
  \end{displaymath}

  We say the adjunction is \emph{fibred} if \(F\) preserves cartesian maps.
\end{defn}

\begin{rmk}
  We can view the above definition as an instance of adjunction for
  \(2\)-categories, by considering the \(2\)-category whose underlying
  \(1\)-category is \(\cats/\BB\) and with \(2\)-cells consisting of
  pointwise vertical natural transformations. See
  e.g. \cite{kellystreettwocategories}.
\end{rmk}

Often in this paper we will switch between fibred adjoints,
adjoints over a fibration and ordinary adjoints in
categories. The lemmas below make it clear when these turn out to be
equivalent, allowing us to drop the distinction between the different
definitions.

\begin{lemma}
  \label{lem:particalrightadjfib}
  Suppose we are given an adjunction \(F \dashv G\) over \(\BB\). Then
  \(G\) preserves cartesian maps.

  Moreover, suppose \(F\) has a partial right adjoint \(G\) over
  \(\BB\) defined at an object \(X\) of \(\DD\) over \(I \in \BB\). If
  \(f : X' \to X\) is cartesian, then \(G\) is also defined at \(X'\)
  and the map \(G(f) : G X' \to G X\) is cartesian.
\end{lemma}

\begin{proof}
  By definition of partial right adjoint over \(\BB\), we have an
  object \(F X\) in \(\EE\) and a vertical map
  \(\epsilon_X : F G X \to X\) which is terminal in
  \((F \downarrow X)\). We define \(\sigma : I \to J\) to be
  \(q(f)\). We then have a diagram in \(\DD\) given by the solid lines
  below.
  \begin{displaymath}
    \begin{tikzcd}
      F (\sigma^\ast(G X)) \ar[rr, "F(\bar{\sigma}(G X))"]
      \ar[dr, dotted, "\epsilon_{X'}"] & & F G X
      \ar[dr, bend left, "\epsilon_X"] & \\
      & X' \ar[rr, "f"] & & X
    \end{tikzcd}
  \end{displaymath}
  This gives us a unique vertical map \(\epsilon_{X'}\) as in the
  dotted line above making a commutative square. We verify that this
  map is terminal in \((F \downarrow X')\).

  Suppose we have an object \(Y\) of \(\EE\) and map \(h : F Y \to
  X'\). By composing with \(f\) we have an object of \((F \downarrow
  X)\) and so a unique map \(t : Y \to G X\) making the commutative
  square below
  \begin{displaymath}
    \begin{tikzcd}
      F Y \ar[r, "F t"] \ar[d] & F G X \ar[d, "\epsilon_X"] \\
      X' \ar[r, "f"] & X
    \end{tikzcd}
  \end{displaymath}
  Now using the fact that \(\bar{\sigma}\) is cartesian we get the
  dotted map \(s\) in the commutative diagram below which is unique
  making the diagram commute and such that \(q(s) = p(h)\).
  \begin{displaymath}
    \begin{tikzcd}
      Y \ar[drr, bend left] \ar[dr, "s", dotted] & & \\
      & \sigma^\ast(G X) \ar[r] & G X
    \end{tikzcd}
  \end{displaymath}
  Finally the fact that \(\epsilon_{X'} \circ F s = h\) follows from
  the fact that \(f\) is cartesian.
\end{proof}

\begin{rmk}
  The first part of the above lemma is a folklore result, that appears
  as \cite[Exercise 1.8.5]{jacobs} for instance.
\end{rmk}

\begin{lemma}
  Suppose we are given fibrations $p : \DD \to \BB$ and $q : \EE \to
  \BB$ and a functor $G : \DD \to \EE$ over $\BB$ such that $G$ has a
  left adjoint as a functor in $\cats$. Then $G$ has a left adjoint
  over $\BB$ if and only if it preserves cartesian maps.
\end{lemma}

\begin{proof}
  The implication \((\Rightarrow)\) follows from Lemma~\ref{lem:particalrightadjfib}.
  
  We show the implication $(\Leftarrow)$. Let $E$ be an object of
  $\EE$. We have an initial object of $(E \downarrow G)$, say
  $\eta_E : E \to G F E$. Say that $\eta_E$ lies over a map $\sigma$
  in $\BB$. We have a cartesian map
  $\bar{\sigma}(FE) : \sigma^\ast(FE) \to F E$. Since $G$ is fibred,
  $G(\bar{\sigma}(FE))$ is also cartesian and lies over
  $\sigma$. Hence there is a unique vertical map
  $\eta'_E : E \to G(\sigma^\ast(FE))$ making a commutative triangle,
  as illustrated below.
  \begin{displaymath}
    \begin{tikzcd}
      E \ar[drr, bend left, "\eta_E"] \ar[dr, dotted, "\eta'_E"] & & \\
      & G(\sigma^\ast(FE)) \ar[r] & G F E
    \end{tikzcd}
  \end{displaymath}

  By the initiality of $\eta_E$ in $(E \downarrow G)$, we have a
  unique map $t : F E \to \sigma^\ast(FE)$, as below.
  \begin{equation*}
    \begin{tikzcd}
      E \ar[r, "\eta_E"] \ar[dr, swap, "\eta'_E"]
      & G F E \ar[d, dotted, "G(t)"] & F E \ar[d, dotted, "t"] \\
      & G (\sigma^\ast(F E)) & \sigma^\ast(F E)
    \end{tikzcd}
  \end{equation*}

  Write $\tau$ for the map $p(t)$ in $\BB$. Applying $q$ to the left
  hand diagram above, we see $\tau \circ \sigma = 1_{q(E)}$.

  We can view both $t$ and $\bar{\sigma}(F E)$ as morphisms in $(E
  \downarrow G)$ and then compose them to get a morphism from the
  object $\eta_E : E \to G F E$ to itself, as illustrated below.
  \begin{equation*}
    \begin{tikzcd}
      & G F E \ar[d, "G(t)"] & F E \ar[d, "t"] \\
      E \ar[ur, "\eta_E"] \ar[r, "\eta'_E" description]
      \ar[dr, swap, "\eta_E"] &
      G(\sigma^\ast(F E)) \ar[d, "G(\bar{\sigma}(F E))"] &
      \sigma^\ast(F E) \ar[d, "\bar{\sigma}(F E)"]  \\
      & G F E & F E
    \end{tikzcd}
  \end{equation*}
  By initiality this morphism can only be the identity on $\eta_E$. We
  conclude that $\bar{\sigma}(F E) \circ t = 1_{F E}$, and so,
  applying $p$ we see $\sigma \circ \tau = 1_{p(F E)}$. We can now see
  that $\sigma$ is an isomorphism with inverse $\tau$. It follows that
  $\bar{\sigma}(F E)$ is also an isomorphism, and so
  $\eta'_E : E \to G(\sigma^\ast(F E))$ is initial in $(E \downarrow
  G)$. By applying this for each object $E$, we can construct a left
  adjoint $F'$ that strictly commutes with $p$ and $q$ with a vertical
  unit $\eta'_E$.
\end{proof}

We can immediately deduce the following proposition.

\begin{prop}
  \label{prop:monadcatcatover}
  Suppose we are given a fibred functor $U : \DD \to \EE$ between
  Grothendieck fibrations. Then $U$ is strictly monadic as a functor
  in $\cats$ if and only if it is strictly monadic for a monad over
  $\BB$.
\end{prop}


\section{Notions of structure and definability}
\label{sec:noti-fibr-struct}

We now give the definition of notion of structure on a fibration, which is
essentially equivalent to Shulman's notion of fibred structure
\cite[Section 3]{shulmaninftytopunivalence}.

\begin{defn}
  Given a fibred functor $\chi$, we write
  $\cart(\chi) : \cart(\DD) \to \cart(\EE)$ for the restriction to
  cartesian maps.
\end{defn}

\begin{defn}
  Suppose we are given a fibred functor between Grothendieck
  fibrations, as illustrated below.
  \begin{displaymath}
    \begin{tikzcd}
      \DD \ar[dr, swap, "p"] \ar[rr, "\chi"] & & \EE \ar[dl, "q"] \\
      & \BB &
    \end{tikzcd}
  \end{displaymath}

  We say $\chi$ \emph{creates cartesian lifts} if $\cart(\chi)$ is a
  discrete fibration.

  We say a \emph{fibred notion of structure}, or just \emph{notion of
    structure} on a Grothendieck fibration $q \colon \EE \to \BB$ is
  another fibration $p \colon \DD \to \BB$ together with a functor
  $\chi$ from $\DD$ to $\EE$ that creates cartesian lifts.
\end{defn}

We can understand the definition of notion of structure through the
following proposition, whose proof is left as an exercise for the
reader.
\begin{prop}
  A fibred functor \(\chi : \DD \to \EE\) creates cartesian lifts if
  and only if for each \(I \in \BB\) the restriction of
  \(\chi_I : \DD_I \to \EE_I\) to isomorphisms is a discrete fibration.
\end{prop}



An object \(X\) of \(\EE_I\), is typically an \(I\)-indexed family
\((X_i)_{i \in I}\) in some sense. We think of objects the fibre
\(\chi_I^{-1}(X)\) as a choice of \emph{structure} on each object \(X_i\) in the
family. We think of morphisms in \(\DD\) as families of structure
preserving \emph{homomorphisms}. The condition of creating cartesian
lifts says that given an isomorphism \(f : X_i \cong Y_i\) and a structure
on \(Y_i\) we can find a unique structure on \(X_i\) making \(f\) a
structure preserving isomorphism.

\begin{defn}
  We say a notion of structure, \(\chi\) is \emph{definable} if
  $\cart{\chi}$ has a right adjoint, as an ordinary functor between
  categories (not over $\BB$).
\end{defn}

\begin{rmk}
It might look a little strange to only require the right adjoint to
exist in categories, without even needing it to commute with the
fibrations. We observe however, that this is also what happens with
the notion of \emph{multi left adjoint} \cite[Section 3]{diersfum}.
Namely, any functor between categories \(\CC\) and \(\DD\) corresponds
to a unique fibred functor between set indexed family fibrations. A
multi left adjoint to a functor \(F : \CC \to \DD\) is precisely a
left adjoint to the corresponding functor
\(\fmly(\CC) \to \fmly(\DD)\). A right adjoint to the functor
\(\fmly(\CC) \to \fmly(\DD)\) in categories can be seen as a ``multi
right adjoint'' following the same idea as multi left adjoints.
\end{rmk}

To better understand the definition of definability we give
some alternative versions of the definition below.

\begin{defn}
  \label{def:locpresheaf}
  Given $X \in \EE$ we write $\bar{\chi}_X$ for the presheaf on
  $\BB/q(X)$ defined as follows. Given an object $\sigma : I \to q(X)$
  in $\BB/q(X)$, we define $\bar{\chi}_X(\sigma)$ to be set of objects
  of $\chi^{-1}(\sigma^\ast(X))$.
\end{defn}

\begin{lemma}
  \label{lem:partreprequiv}
  Let $X$ be an object of $\EE$. Then the right adjoint to
  $\cart(\chi)$ is defined at $X$ if and only if the presheaf
  $\bar{\chi}_X$ is representable.
\end{lemma}

\begin{proof}
  Note that we have an equivalence of categories $\BB/q(X) \simeq
  \cart(\DD)/X$ and that $\bar{\chi}_X$ corresponds to the discrete
  fibration obtained by pulling back along the composition of the
  equivalence and $\dom : \cart(\EE) / X \to \cart(\EE)$ as illustrated below.

  \begin{equation*}
    \begin{tikzcd}
      \cdot \ar[r, "\simeq"] \ar[d] \pbcorner & (\cart(\chi) \downarrow X)
      \ar[r] \ar[d] \pbcorner & \cart(\DD) \ar[d] \\
      \BB / q(X) \ar[r, "\simeq"] & \cart(\EE)/X \ar[r] & \cart(\EE)
    \end{tikzcd}
  \end{equation*}

  Hence we can apply Lemma \ref{lem:repriffpartialra} together with
  the equivalence in the left hand square above.
\end{proof}

We note that when \(\BB\) is a presheaf category we can describe the
representing objects for \(\bar{\chi}_X\) explicitly, as follows. We
emphasise however that even though we can provide a concrete
description, it can still happen that the maps constructed are not
representing for \(\bar{\chi}_X\) and in this case \(\bar{\chi}_X\) is
simply not representable at all.
\begin{thm}
  \label{thm:locrepinpshbase}
  Suppose that the base category \(\BB\) is a presheaf category and
  that \(\chi^{-1}(\{X\})\) is a set for each object \(X\). Then for
  each object \(X\), we can construct an object \(J\) and map
  \(J \to p(X)\) such that if \(\bar{\chi}_X\) is representable, then
  it can be represented by \(\sigma : J \to p(X)\).

  Furthermore, we can construct a natural transformation from
  \(\bar{\chi}_X\) to \(\yoneda \sigma\) which is an isomorphism
  precisely when \(\bar{\chi}_X\) is representable.
\end{thm}

\begin{proof}
  Suppose that \(\BB = \psh{C}\) for some small category
  \(\mathcal{C}\). For each \(c \in \mathcal{C}\) the elements of
  \(J(c)\) correspond to maps \(\yoneda(c) \to J\). We can think of
  each such map as pair consisting of a map
  \(i : \yoneda(c) \to p(X)\) together with a map \(\yoneda(c) \to J\)
  making a commutative triangle with the map \(J \to p(X)\) that we
  have yet to define. However, we know that such commutative triangles
  must correspond precisely to objects in
  \(\chi^{-1}(\{i^\ast(X)\})\). Hence we can just define \(J(c)\) to
  consist of pairs \(i, D\) where \(i : \yoneda(c) \to p(X)\) and
  \(D \in \chi^{-1}(\{i^\ast(X)\})\).
\end{proof}

\begin{prop}
  \label{prop:repdefns}
  The following are equivalent.

  \begin{enumerate}
  \item $\bar{\chi}_X$ is representable for every $X$.
  \item $\cart(\chi)$ has a right adjoint.
  \item $\cart(\chi)$ is comonadic.
  \end{enumerate}
\end{prop}

\begin{proof}
  We have shown \((1 \Leftrightarrow 2)\) in
  Lemma~\ref{lem:partreprequiv}. The implication \((2 \Rightarrow 3)\)
  is directly from Lemma~\ref{lem:discfibra}.
\end{proof}

The relation between local representability and colimits was pointed
out by Shulman in \cite[Proposition
3.18(iii)]{shulmaninftytopunivalence}. We observe that this can be seen as
an instance of comonadic functors creating colimits:
\begin{lemma}
  Suppose we are given a notion of structure
  $\chi : \DD \to \EE$. If $\chi$ is definable, then
  $\cart(\chi)$ (strictly) creates colimits.
\end{lemma}

\begin{proof}
  If \(\chi\) is definable, then \(\cart(\chi)\) is comonadic by
  Proposition \ref{prop:repdefns}, and comonadic functors create
  colimits.
\end{proof}

We finish this section with a couple of useful lemmas that will be
used later. The first says that in one sense the right adjoints
witnessing definability are automatically stable under pullback. The
second says that definable notions of structure are stable under
pullback in the category of fibrations.

\begin{lemma}
  \label{lem:pblocrepr}
  Suppose that $\BB$ has pullbacks and $\chi$ creates cartesian
  lifts. If $\bar{\chi}_X$ is representable for $X \in \EE_I$, then so
  is $\bar{\chi}_{\sigma^\ast(X)}$ for $\sigma : J \to I$.

  Moreover, if $\bar{\chi}_X$ is representable by $\tau : K \to I$,
  then the representing object for $\bar{\chi}_{\sigma^\ast(X)}$ can
  be explicitly described as the pullback $\sigma^\ast(\tau)$.

  \begin{displaymath}
    \begin{tikzcd}
      & \sigma^\ast(X) \ar[rr, "\bar{\sigma}(X)"]
      & & X \ar[dd, dash, dotted] & \EE \ar[d, "q"] \\
      \sigma^\ast(X) \ar[rr] \ar[dr, swap, "\sigma^\ast(\tau)"]
      \ar[drrr, phantom, very near
      start, "\lrcorner"] & & K \ar[dr, "\tau"] & & \BB \\
      & J \ar[rr, "\sigma"] \ar[from=uu, crossing over, dash, dotted] & & I
    \end{tikzcd}
  \end{displaymath}
\end{lemma}

\begin{proof}
  Suppose we are given an object of $\BB/J$ of the form
  $\rho : L \to J$. Then maps from $\rho$ to $\sigma^\ast(\tau)$ in
  $\BB/J$ correspond naturally to maps from $\sigma \circ \rho$ in
  $\BB/I$. These correspond naturally to elements of
  $\bar{\chi}_X(\sigma \circ \rho)$. Note however that we have natural
  isomorphisms
  $\bar{\chi}_X(\sigma \circ \rho) \cong
  \bar{\chi}_{\sigma^\ast(X)}(\rho)$, since we can lift the
  isomorphism
  $(\sigma \circ \rho)^\ast(X) \cong \rho^\ast(\sigma^\ast(X))$ to a
  bijection between the objects of
  $\chi^{-1}((\sigma \circ \rho)^\ast(X))$ and those of
  $\chi^{-1}(\rho^\ast(\sigma^\ast(X)))$ using the assumption that
  $\chi$ creates cartesian lifts, and moreover the bijections are
  natural in $\rho$. We deduce that there is a natural correspondence
  between morphisms from $\rho$ to $\sigma^\ast(\tau)$ in $\BB/J$ and
  elements of $\bar{\chi}_{\sigma^\ast(X)}(\rho)$, giving us the
  required isomorphism between $\bar{\chi}_{\sigma^\ast(X)}$ and the
  representable on $\sigma^\ast(\tau)$.
\end{proof}

\begin{lemma}
  \label{lem:cclpbstab}
  Suppose we are given a strict pullback diagram in fibrations over
  $\BB$, as below.
  \begin{equation*}
    \begin{tikzcd}
      \CC \times_\EE \DD \ar[rr] \ar[d, swap, "\rho^\ast(\chi)"]
      \ar[drr, phantom, very near start, "\lrcorner"]
      & &
      \DD \ar[d, "\chi"] \\
      \CC \ar[rr, "\rho"] \ar[dr] & & \EE \ar[dl] \\
      & \BB &
    \end{tikzcd}
  \end{equation*}
  If $\chi$ creates cartesian lifts, then so does $\rho^\ast(\chi)$.
\end{lemma}

\begin{proof}
  Note that a map $(f, g)$ in $\CC \times_\EE \DD$ is cartesian if and
  only if $f$ is cartesian in $\CC$ and $g$ is cartesian in
  $\DD$. Hence we have a pullback diagram in $\cats$ as below.
  \begin{equation*}
    \begin{tikzcd}
      \cart(\CC \times_\EE \DD) \ar[r] \ar[d, swap,
      "\cart(\rho^\ast(\chi))"] \pbcorner & \cart(\DD) \ar[d, "\cart(\chi)"]
      \\
      \cart(\CC) \ar[r] & \cart(\EE)
    \end{tikzcd}
  \end{equation*}

  However, discrete fibrations are stable under pullback, so if
  $\cart(\chi)$ is a discrete fibration, then so is
  $\cart(\rho^\ast(\chi))$.
\end{proof}

\section{Some examples of notions of structure}
\label{sec:some-examples}

\subsection{Full notions of structure}
\label{sec:examples-full-fibred}

\begin{defn}
  Let $q : \EE \to \BB$ be a fibration and $\mathcal{D} \subseteq \EE$
  a class of objects. We say $\mathcal{D}$ is \emph{closed under
    substitution} if whenever $X \to Y \in \EE$ is cartesian and
  $Y \in \mathcal{D}$ we also have $X \in \mathcal{D}$.
\end{defn}

Note that a class of objects is closed under substitution if and only
if the corresponding inclusion of a full subcategory
$\DD \hookrightarrow \EE$ is a notion of structure. Following Shulman,
we refer to notions of structure of this form as \emph{full}. In this
case definability recovers the definition of definability due to
B\'{e}nabou \cite{benaboufibcat}.

\begin{defn}[B\'{e}nabou]
  We say a class of objects closed under substitution is
  \emph{definable} if the corresponding full notion of fibred
  structure is definable.
\end{defn}

\begin{prop}
  Suppose $\mathcal{D} \subseteq \EE$ is a definable class of
  objects. For each $X \in \EE$, the representing object for
  $\bar{\chi}_X$ is a monomorphism as a map in $\BB$.
\end{prop}

\begin{proof}
  Note that for full notions of structure each presheaf
  $\bar{\chi}_X$ is subterminal. The Yoneda embedding reflects
  subterminal objects, so the representing object is subterminal as an
  object of $\BB / q(X)$, which precisely says it is a monomorphism as
  a map in $\BB$.
\end{proof}

This tells us that the representing object for $\bar{\chi}_X$ is a
subobject of $q(X)$. Unfolding the definitions, it is the largest
subobject $\sigma : I \rightarrowtail q(X)$ such that $\sigma^\ast(X)$
belongs to $\mathcal{D}$.

We give some basic examples of full notions of structure and
definability to illustrate the idea. See e.g. \cite[Section
9.6]{jacobs}, \cite[Section 12]{streicherfibcat} or \cite[Section
B1.3]{theelephant} for a more complete account.

\begin{example}
  \label{ex:ptwisefull}
  Given any category $\CC$ and any class of objects
  $\mathcal{D} \subseteq \CC$, we can define a full notion of
  structure on the fibration of set or category indexed families on
  $\CC$. A family $(X_i)_{i \in I}$ belongs to the class if $X_i$ is
  an element of $\mathcal{D}$ for every $i \in I$.

  Classes of this form are always definable, assuming we have the
  axiom of full separation in the set theory where we are working.
  Given a set indexed
  family $X := (X_i)_{i \in I}$, the representing object of
  $\bar{\chi}_X$ is the subset of $I$ defined as
  $\{i \in I \;|\; X_i \in \mathcal{D} \}$.

  For category indexed families, the representing object is a full
  subcategory with set of objects defined as for set indexed families.
\end{example}

\begin{prop}
  Let $\BB$ be a regular category. Then regular epimorphisms form a
  definable class for the codomain fibration $\cod : \BB^\to \to \BB$.
\end{prop}

\begin{proof}
  First recall that in a regular category, regular epimorphisms are
  stable under pullback, which precisely says they are closed under
  substitution in the codomain fibration.

  Now given $f : X \to Y$, we have an image factorisation
  $X \twoheadrightarrow \im(f) \rightarrowtail Y$. We note that the
  square below is a pullback, e.g. by directly verifying the universal
  property.
  \begin{displaymath}
    \begin{tikzcd}
      X \ar[r, equal] \ar[d, two heads] \pbcorner & X \ar[d, "f"] \\
      \im(f) \ar[r, tail] & Y
    \end{tikzcd}
  \end{displaymath}

  Given any object $Z$ and any map $h : Z \to Y$ we have the
  commutative square below.
  \begin{displaymath}
    \begin{tikzcd}
      h^\ast(X) \ar[r] \ar[d, "h^\ast(f)"] & X \ar[r, two heads] & \im(f) \ar[d,
      tail] \\
      Z \ar[rr] & & Y
    \end{tikzcd}
  \end{displaymath}
  When $h^\ast(f)$ is a regular epimorphism, we get a unique diagonal
  filler $Z \to \im(f)$, which witnesses $h^\ast(f)$ an a pullback of
  $X \twoheadrightarrow \im(f)$, as below.
  \begin{equation*}
    \begin{tikzcd}
      h^\ast(X) \ar[r] \ar[d] \pbcorner & X \pbcorner \ar[d]
      \ar[r, equal] & X \ar[d] \\
      Z \ar[r] & \im(f) \ar[r] & Y
    \end{tikzcd}
  \end{equation*}
  Conversely, given a pullback square as in the left hand square
  above, we can deduce that $h^\ast(f)$ is a regular epimorphism.
\end{proof}

More generally, given a pullback stable factorisation system on a
category $\BB$, the left class will always give a definable class with
respect to $\cod : \BB^\to \to \BB$.

\begin{example}
  Let $q : \EE \to \BB$ be a fibration such that reindexing preserves
  any terminal objects that exist. E.g. set or category indexed
  families on a category, or any codomain fibration on a category with
  finite limits. Then the class objects $X$ of $\EE$ that are
  terminal in their fibre category is closed under substitution.
\end{example}

\begin{example}
  Let $q : \EE \to \BB$ be a fibration such that each fibre category
  has a terminal object and reindexing preserves monomorphisms and
  terminal objects. Again this includes set or category indexed
  families and any codomain fibration. An object $X$ of $\EE_I$ is
  \emph{subterminal} if the unique map $X \to 1_I$ is a
  monomorphism. Then the class of subterminal objects is closed under
  substitution.

  As a special case the subterminal objects in $\cod : \BB^\to \to
  \BB$ are precisely the objects that are monomorphisms in $\BB$.
\end{example}

\begin{prop}
  Suppose that $\BB$ is a Heyting category. Then monomorphisms are a
  definable class in $\cod : \BB^\to \to \BB$.
\end{prop}

\begin{proof}
  The intuitive idea is that subterminal objects can be described
  within the internal language of a Heyting category. An object is
  subterminal if any two elements of it are equal, namely if it
  satisfies the following sentence: \(\forall x, y \in X\; x = y\). In
  the argument below we expand out the preceding sentence to an
  explicit categorical description, and check that it works.
  
  Suppose we are given a map $f : X \to Y$. We have a diagonal map
  $\Delta_X : X \rightarrowtail X \times_Y X$. Write $p$ for the
  canonical map $X \times_Y X \to Y$. We then have a monomorphism
  $\forall_p \, \Delta_X : \forall_p \, X \rightarrowtail Y$. We check
  that this does give a representing object for $\bar{\chi}_f$.

  First note that a map $h : Z \to Y$ factors through
  $\forall_p\,\Delta$ if and only if
  $\top \leq h^\ast(\forall_p\,\Delta_X)$ in the lattice of subobjects
  of $Z$. This is the case precisely when $\exists_h\,\top \leq
  \forall_p\,\Delta_X$ in subobjects of $Y$, which holds when
  $p^\ast(\exists_h\,\top) \leq \Delta_X$ in subobjects of $X \times_Y
  X$. Since image factorisation is stable under pullback, we have
  $p^\ast(\exists_h\,\top) \cong \exists_{p^\ast(h)}\,\top$. Hence 
  $p^\ast(\exists_h\,\top) \leq \Delta_X$ precisely when $\top \leq
  (p^\ast(h))^\ast(\Delta_X)$ in subobjects of $h^\ast(X) \times_Z
  h^\ast(X)$. However, one can calculate that
  $(p^\ast(h))^\ast(\Delta_X)$ is exactly the diagonal map
  $\Delta_{h^\ast(X)} : h^\ast(X) \to h^\ast(X) \times_Z h^\ast(X)$,
  which is equal to $\top$ precisely when $h^\ast(f)$ is a monomorphism.
\end{proof}

\begin{example}
  Suppose we are given a fibration $q : \EE \to \BB$. Then we can
  define a second fibration of vertical arrows $V(\EE) \to \BB$. In
  any case the class of isomorphisms is a full notion of fibred
  structure on $V(\EE)$. If reindexing preserves monomorphisms, then
  they are also a full notion of structure.

  If $q$ is locally small, then both of these are definable
  classes.
\end{example}

\begin{thm}[B\'{e}nabou]
  Let $\BB$ be a topos, together with a local operator $j : \Omega \to
  \Omega$. The following classes of maps are definable classes of objects
  with respect to the codomain fibration on $\BB$.
  \begin{enumerate}
  \item Families of $j$-separated objects.
  \item Families of $j$-sheaves.
  \end{enumerate}
\end{thm}

\begin{example}
  \label{ex:wfsfns}
  Let $(\mathcal{L}, \mathcal{R})$ be a weak factorisation system on a
  category $\BB$. Then the right class $\mathcal{R}$ gives us a full
  notion of structure on the codomain fibration on $\BB$. We say a wfs
  is \emph{definable} if the corresponding full notion of structure on
  $\cod : \BB^\to \to \BB$ is definable.

  As remarked by Shulman \cite[Example
  3.17]{shulmaninftytopunivalence} a wfs is definable
  if it is cofibrantly generated by a set of maps with representable
  codomain, assuming the axiom of choice. We will see in Section
  \ref{sec:some-non-definable} that the axiom of choice is strictly
  necessary.
\end{example}

\subsection{Comprehension schemes}
\label{sec:compr-schem}

The idea of definability appears again in Johnstone's notion of
\emph{comprehension schemes}. Although we make some adjustments to fit
with the general theory of notions of structure, the idea essentially
appears in \cite[Section B1.3]{theelephant}. In particular, in the
definition below we include a requirement of isomorphism on objects to
satisfy the definition of notion of structure. This is not required by
Johnstone, who instead refers to it as as a special case where the
definition ``works best.''\footnote{Alternatively one can make
  Johnstone's definition better behaved by working in univalent
  \(\infty\)-categories \cite{stenzelcc}.}

\begin{prop}
  Let $q : \EE \to \BB$ be any fibration and
  $F : \mathcal{C} \to \mathcal{D}$ an internal functor between
  internal categories in $\BB$. If $F$ is an isomorphism on objects,
  then the corresponding
  $F^\ast : \EE^\mathcal{D} \to \EE^\mathcal{C}$ between fibrations of
  diagrams defined by precomposing with $F$, creates cartesian lifts.
\end{prop}

\begin{proof}
  Without loss of generality, $\mathcal{C}$ and $\mathcal{D}$ have the
  same object of objects and $F_0$ is the identity. We can write out
  the remaining data for the internal categories and internal functor
  as the following commutative diagram.
  \begin{equation*}
    \begin{tikzcd}
      \mathcal{C}_1 \ar[rr, "F_1"] \ar[dr, shift left, "s"]
      \ar[dr, swap, shift right, "t"]
      & & \mathcal{D}_1 \ar[dl, shift left, "v"] \ar[dl, shift right, swap,
      "u"] \\
      & \mathcal{C}_0 &
    \end{tikzcd}
  \end{equation*}

  We can view a pair consisting of an object $Y$ of $\EE^\mathcal{D}$ and
  a cartesian map into $F^\ast(Y)$ in $\EE^\mathcal{C}$ as the solid
  lines in the upper diagram below, where the horizontal maps are
  cartesian over the maps in the lower square.
  \begin{displaymath}
    \begin{tikzcd}[sep = 1.5em]
      & s^\ast(Y) \ar[rr] \ar[dd] & & u^\ast(Y) \ar[dd] \\
      s^\ast(X) \ar[rr, crossing over] \ar[ur] \ar[dd, dotted] & & u^\ast(X)  \ar[ur] & \\
      & t^\ast(Y) \ar[rr] & & v^\ast(Y) \\
      t^\ast(X) \ar[rr] \ar[ur] & & v^\ast(X) \ar[ur] \ar[from = uu, crossing
      over]&
    \end{tikzcd}
  \end{displaymath}
  \begin{displaymath}
    \begin{tikzcd}[sep = 1.5em]
      & J \times \mathcal{C}_1 \ar[rr, "J \times F_1"] & & J \times \mathcal{D}_1 \\
      I \times \mathcal{C}_1 \ar[ur] \ar[rr, "I \times F_1"] & & I \times
      \mathcal{D}_1 \ar[ur] &
    \end{tikzcd}
  \end{displaymath}

  Cartesian lifts correspond precisely to vertical morphisms
  completing the diagram to a commutative cube, as in the dotted map
  above. However, such maps are uniquely determined by the universal
  property of the cartesian maps in the above diagram.
\end{proof}

\begin{defn}[Johnstone]
  Given an isomorphism on objects internal functor $F$ in $\BB$, we
  say $q : \EE \to \BB$ \emph{satisfies the comprehension scheme for
    $F$} if the notion of structure $F^\ast$ is definable.
\end{defn}

\begin{example}
  Let $\BB$ be a category with finite limits and finite
  coproducts. Define internally in $\BB$ the inclusion functor from
  the discrete category on $2$ objects, $2$, to the category with two
  objects and a morphism from one to the other, denoted $\cdot \to
  \cdot$.

  We can explicitly describe the diagram category $\EE^2$ as the
  pullback $\EE \times_\BB \EE$ and the diagram category $\EE^\to$ as
  the category of vertical maps $V(\EE)$. The functor $F^\ast : V(\EE)
  \to \EE \times_\BB \EE$ sends a vertical map to its domain and
  codomain.

  We see that in this case a ``structure'' on a pair of objects $X, Y$
  in the same fibre category $\EE_I$ is a vertical map from $X$ to
  $Y$.

  This notion of structure is definable if and only
  if $q : \EE \to \BB$ is locally small.
\end{example}

\begin{example}
  Let $\BB$ be category with finite limits and finite coproducts. We
  can construct internally in $\BB$ the category with two objects and a
  map between them $\cdot \to \cdot$ as well as the category with two
  objects and two maps between them $\cdot \rightrightarrows
  \cdot$. Furthermore, we can define the unique functor $F$ from
  $\cdot \rightrightarrows \cdot$ to $\cdot \to \cdot$ that is the
  identity on objects (and ``collapses'' the two morphisms).

  As before, we can explicitly describe $\EE^\to$ as the category of
  vertical arrows. We explicitly describe $\EE^\rightrightarrows$ as
  the category of pairs of vertical arrows with the same domain and
  same codomain. We then have that $F^\ast$ is the inclusion of the
  full subcategory of $\EE^\rightrightarrows$ of objects where the two
  arrows in the pair are equal. If $q$ satisfies the comprehension
  scheme for $F$ we say it has \emph{definable equality}.
\end{example}

\begin{thm}[Johnstone]
  A fibration $q : \EE \to \BB$ is locally small if and only if it
  satisfies the comprehension scheme for all isomorphism on objects
  internal functors $F : \mathcal{C} \to \mathcal{D}$ in $\BB$.
\end{thm}

\begin{proof}
  See the remark after \cite[Lemma B1.3.15]{theelephant}.
\end{proof}

\subsection{(Co)Algebraic notions of structure}
\label{sec:coalg-noti-fibr}

We finally turn to the main source of motivating examples for this
paper. These observations already appear in \cite[Sections 4.3 and
5.1]{swanliftprob} and are minor variants of standard material, but we
repeat them below for reference.

In the below, we assume we are given an arbitrary Grothendieck
fibration $q : \EE \to \BB$.

\begin{defn}
  An \emph{endofunctor over $\BB$} is a functor $T : \EE \to \EE$ such
  that $q \circ T = T$.

  A \emph{pointed endofunctor over $\BB$} is an endofunctor $T$ over
  $\BB$ together with a natural transformation $\eta : 1_\EE
  \Rightarrow T$ such that $\eta_X : X \to T X$ is a vertical map for
  each $X \in \EE$.

  A \emph{monad over $\BB$} is a pointed endofunctor $(T, \eta)$ over
  $\BB$ together with a natural transformation $\mu : T^2 \Rightarrow
  T$ such that $\mu_X : T(TX) \to T X$ is vertical for each $X \in
  \EE$, and such that $\eta$ and $\mu$ satisfy the usual monad
  laws, displayed below for reference.
  \begin{displaymath}
    \begin{gathered}
      \begin{tikzcd}
        T \ar[r, "\eta_T"] \ar[dr, equal] & T^2 \ar[d, "\mu"] & 
        T \ar[l, swap, "T \eta"] \ar[dl, equal] \\
        & T &
      \end{tikzcd}
    \end{gathered}
    \qquad
    \begin{gathered}
      \begin{tikzcd}
        T^3 \ar[r, "T \mu"] \ar[d, "\mu_T"] & T^2 \ar[d, "\mu"] \\
        T^2 \ar[r, "\mu"] & T
      \end{tikzcd}
    \end{gathered}
  \end{displaymath}

  We say an endofunctor over $\BB$, $T : \EE \to \EE$ is \emph{fibred}
  if it preserves cartesian maps, and pointed endofunctors and monads
  are fibred if their underlying endofunctors are.

  We dually define (fibred) copointed endofunctors and comonads.
\end{defn}

We emphasise that in this definition monads over $\BB$ are not
necessarily fibred (i.e. do not necessarily preserve cartesian maps)
and in fact many natural examples of algebraic weak factorisation
systems (to be covered in section \ref{sec:algebr-weak-fact}) are not.

Monads over $\BB$ can be seen as a special case of the $2$-categorical
definition of monad \cite{streetmonads} by working in the $2$-category
whose underlying $1$-category is $\cats/\BB$ and whose $2$-cells are
pointwise vertical natural transformations. Fibred monads are monads
in the usual $2$-category of fibrations over $\BB$ (see
e.g. \cite[Section 1.7]{jacobs}).

\begin{defn}
  Let $T$ be an endofunctor over $\BB$ and $X$ an object of $\EE$. A
  \emph{$T$-algebra structure} on $X$ is a vertical map
  $s : T X \to X$. A \emph{$T$-algebra} is an object of $\EE$ together
  with $T$-algebra structure. This defines a category $\xalg{T}$,
  together with a forgetful functor $\upsilon : \xalg{T} \to \EE$.

  Let $(T, \eta)$ be a pointed endofunctor over $\BB$. An algebra
  structure on $X \in \EE$ is a (necessarily vertical) map $s : T X
  \to X$ that satisfies the usual unit law, displayed below for
  reference.
  \begin{displaymath}
    \begin{tikzcd}
      X \ar[r, "\eta_X"] \ar[dr, equal] & T X \ar[d, "s"] \\
      & X
    \end{tikzcd}
  \end{displaymath}
  We similarly define the category of $(T, \eta)$-algebras, again
  written as $\xalg{T}$ when $\eta$ is clear from the context.

  Let $(T, \eta, \mu)$ be a monad over $\BB$. An algebra structure on
  an object $X \in \EE$ is an algebra structure on the underlying
  pointed endofunctor $s : T X \to X$ that additionally satisfies the
  usual multiplication law, displayed below for reference.
  \begin{displaymath}
    \begin{tikzcd}
      T(T X) \ar[r, "T s"] \ar[d, "\mu_X"] & T X \ar[d, "s"] \\
      T X \ar[r, "s"] & X
    \end{tikzcd}
  \end{displaymath}
  We again write the category of algebras as $\xalg{T}$ when $\eta$
  and $\mu$ are clear from the context.

  We dually define categories of \emph{coalgebras} $\xcoalg{M}$ for
  endofunctors, copointed endofunctors and comonads $M$ over $\BB$.
\end{defn}

\begin{lemma}
  \label{lem:algccl}
  Let $T$ be an endofunctor, pointed endofunctor or monad over
  $\BB$. The forgetful functor $\upsilon : \xalg{T} \to \EE$ creates
  cartesian lifts. In particular the composition $\xalg{T} \to \BB$ is
  a Grothendieck fibration.
\end{lemma}

\begin{proof}
  Suppose we are given a cartesian map $f : X \to Y$ in $\EE$ together
  with an algebra structure on $Y$. Together this gives us the solid
  lines in the diagram below, where $f$ and $T f$ lie over the same
  map in $\BB$, and $s$ is vertical.
  \begin{equation*}
    \begin{tikzcd}
      T X \ar[rr, "T f"] \ar[dr, dotted, "t"] & & T Y \ar[dr, bend left, "s"] & \\
      & X \ar[rr, "f"] & & Y
    \end{tikzcd}
  \end{equation*}
  However, since $f$ is cartesian, there is a unique vertical map
  $t : T X \to X$ making a commutative square as in the dotted arrow
  above. This is precisely an algebra structure on $X$ making $f$ a
  homomorphism of algebras. One can check that $f$ remains cartesian
  as a map in $\xalg{T}$.

  Furthermore, one can check that if $s$ satisfies the unit law for a
  pointed endofunctor or the multiplication law for a monad, then so
  does $t$.
\end{proof}

\begin{lemma}
  \label{lem:coalgccl}
  Let $M$ be a fibred endofunctor, copointed endofunctor or comonad
  over $\BB$. The forgetful functor $\upsilon : \xcoalg{M} \to \EE$
  creates cartesian lifts. In particular the composition
  $\xcoalg{M} \to \BB$ is a Grothendieck fibration.
\end{lemma}

\begin{proof}
  Suppose we are given a cartesian map $f : X \to Y$ in $\EE$ together
  with a coalgebra structure on $Y$. Together this gives us the solid
  lines in the diagram below, where $f$ and $M f$ lie over the same
  map in $\BB$, and $s$ is vertical.  
  \begin{equation*}
    \begin{tikzcd}
      X \ar[rr, "f"] \ar[dr, dotted, "t"] & & Y \ar[dr, bend left, "s"] & \\
      & M X \ar[rr, "M f"] & & M Y
    \end{tikzcd}
  \end{equation*}
  Since $M$ is fibred, $M f$ is cartesian, and so there is a unique
  vertical map $t : X \to M X$ making a commutative square as in the
  dotted arrow above. This is exactly a coalgebra structure on $X$
  making $f$ a homomorphism of coalgebras. As before, one can check
  that $f$ remains cartesian as a map in $\xcoalg{M}$ and that $t$
  satisfies counit and comultiplication laws when $s$ does.
\end{proof}

\begin{rmk}
  We emphasise that Lemma \ref{lem:coalgccl} required the
  additional assumption that the comonad is fibred, so Lemmas
  \ref{lem:algccl} and \ref{lem:coalgccl} are not formally dual. The
  dual to Lemma \ref{lem:algccl} tells us that any forgetful
  functor $\xcoalg{M} \to \EE$ creates opcartesian lifts, whereas the
  dual to Lemma \ref{lem:coalgccl} tells us that if an
  endofunctor, pointed endofunctor or monad $T$ preserves opcartesian
  maps, then $\xalg{T} \to \EE$ creates opcartesian lifts.
\end{rmk}

\begin{example}
  \label{ex:ptdobjects}
  Suppose we are given a choice of terminal object $1_I$ for each
  $I \in \BB$ and reindexing preserves terminal objects.

  This defines a fibred endofunctor over $\BB$ by $T(X) :=
  1_{q(X)}$. An algebra structure on an object $X$ is simply a map
  $1_{q(X)} \to X$. We refer to algebras as \emph{pointed objects} and
  write the category of algebras as $\EE_\bullet$. If the notion of
  structure \(\EE_\bullet \to \EE\) is definable, we say the fibration
  \(\EE \to \BB\) \emph{admits comprehension} \cite{lawverecomp}.
\end{example}

\begin{example}
  \label{ex:sections} As a special case of Example \ref{ex:ptdobjects}
  we can consider pointed objects in codomain fibrations. An object of
  \(\BB^\to\) is a map \(f\) of \(\BB\). A point of \(f\) as an object
  of \(\BB^\to\) is then precisely a section of \(f\).
\end{example}

For a set indexed family fibration $\fmly(\CC) \to \sets$, a pointed
object is a family of objects $(C_i)_{i \in I}$ together with a choice
of map $c_i : 1_\CC \to C_i$ in $\CC$ for each $i \in I$.

A pointed object in a codomain fibration $\BB^\to \to \BB$ is an
object of $\BB^\to$, which is a map $f : X \to I$, together with a map
from the identity on $I$ to $f$ in $\BB/I$, which is just a section of
$f$.

\begin{lemma}
  \label{lem:sectlocrepr}
  For any category $\BB$ with pullbacks, the forgetful functor from
  maps with sections to maps, $\upsilon : \BB^\to_\bullet \to \BB^\to$ is
  definable.

  Moreover, for a map $f : X \to I$, the representing object for
  $\bar{\upsilon}_f$ in $\BB / I$ is simply $f$ itself.
\end{lemma}

\begin{proof}
  For any $f : X \to I$ in $\BB^\to$, and any $\sigma : J \to I$,
  sections of $\sigma^\ast(f) : \sigma^\ast(X) \to J$ correspond
  precisely to maps $J \to X$ making a commutative triangle by the
  universal property of the pullback.
  \begin{displaymath}
    \begin{gathered}
      \begin{tikzcd}
        \sigma^\ast(X) \ar[r] \ar[d] \pbcorner & X \ar[d] \\
        J \ar[r, "\sigma"] \ar[u, bend left, dotted] & I
      \end{tikzcd}
    \end{gathered}
    \qquad
    \begin{gathered}
      \begin{tikzcd}
        J \ar[rr, dotted] \ar[dr, "\sigma"] & & X \ar[dl, "f"] \\
        & I &
      \end{tikzcd}
    \end{gathered}
  \end{displaymath}
\end{proof}

We also give a simple non-fibred example. Although it is not an awfs
itself, it illustrates the essential idea why many natural examples of
awfs's are not fibred.

\begin{example}
  Consider the codomain fibration $\cod : \BB^\to \to \BB$. We define
  an endofunctor $T$ as follows. Given $f : X \to I$ we define $T(f)$
  to be the first projection $\pi_0 : I^2 \to I$. We can visualise
  algebra structures on $f$ as functions assigning for each pair
  $i, j : I$ an element $x_{i, j}$ of $X_i$, the fibre of $f$ over
  $i$. This picture can be made precise using the internal language of
  $\BB$.

  Given a map $\sigma : J \to I$, the pullback $\sigma^\ast(\pi_0)$ is
  the projection $J \times I \to J$, not $J^2 \to J$. Hence the
  endofunctor does not preserve pullbacks.

  The forgetful functor $\xalg{T} \to \BB^\to$ still creates cartesian
  lifts, and in particular $\xalg{T} \to \BB$ is a fibration by
  Lemma \ref{lem:algccl}.
\end{example}

\begin{example}
  For an example of a coalgebraic notion of structure, we assume that
  the fibration \(q : \EE \to \BB\) has fibred coproducts and terminal
  objects and consider the endofunctor sending \(X\) to \(1_{q(X)} +
  1_{q(X)}\). A coalgebra structure on \(X\) is a \(2\)-colouring,
  i.e. a partition of \(X\) into two pieces.
\end{example}

\begin{thm}
  \label{thm:fibmonadlocrepr}
  Let $q : \EE \to \BB$ be a locally small fibration, and suppose
  $\BB$ has all finite limits. The forgetful functors from categories
  of (co)algebras for fibred endofunctors, (co)pointed endofunctors
  and (co)monads are all definable.
\end{thm}

\begin{proof}
  We will show this for algebras, the proof for coalgebras being
  similar.\footnote{With care it is also possible to deduce the result
    for coalgebras by duality.}

  Let $P$ be a fibred endofunctor over $\BB$. Fix $X \in \EE$. We need
  to show that the presheaf on $\BB/q(X)$ sending
  $\sigma : I \to p(X)$ to $P$ algebra structures on $\sigma^\ast(X)$
  is representable. The set of algebra structures is by definition the
  hom set $\EE(P(\sigma^\ast(X)), \sigma^\ast(X))$, which is naturally
  isomorphic to $\EE(\sigma^\ast(P(X)), \sigma^\ast(X))$, since $P$ is
  fibred. However, the latter presheaf is representable by the
  characterisations of local smallness in terms of representables.

  Now suppose we are given a fibred pointed endofunctor
  $\eta : 1 \to P$ over $\BB$. We again need to show that the presheaf
  sending $\sigma : I \to p(X)$ to $(P, \eta)$ algebra structures on
  $\sigma^\ast(X)$ is representable. Such an algebra structure is
  precisely a map $f : P(\sigma^\ast(X)) \to \sigma^\ast(X)$ such that
  $f \circ \eta_X = 1_X$. Observe that we can express the set of
  algebra structures as an equalizer in sets of the form
  $\EE(P(\sigma^\ast(X)), \sigma^\ast(X)) \rightrightarrows
  \EE(\sigma^\ast(X), \sigma^\ast(X))$. Hence we can view the presheaf
  of algebra structures as an equalizer in presheaves defined
  pointwise as the preceding equalizer for each $\sigma$. Since
  $\BB/p(X)$ has all finite limits and the Yoneda embedding preserves
  limits we can deduce that the presheaf of algebra structures is
  representable. Namely, the representing object is defined as an
  equalizer in $\BB$ of the following form.
  \begin{displaymath}
    \hom(P(X), X) \rightrightarrows \hom(X, X)
  \end{displaymath}

  Finally, if $(P, \eta, \mu)$ is a monad, we can repeat the argument
  for pointed endofunctors, but now also need to specify the
  multiplication law as well as the unit law. Namely, the representing
  object is constructed as a limit in $\BB/p(X)$ of the following
  form.
  \begin{displaymath}
    \begin{tikzcd}
      & \hom(X, X) \\
      \hom(P(X), X) \ar[ur, shift left] \ar[ur, shift right]
      \ar[dr, shift left] \ar[dr, shift right] & \\
      & \hom(P^2(X), X)
    \end{tikzcd}
  \end{displaymath}
\end{proof}

\subsection{Algebraic weak factorisation systems}
\label{sec:algebr-weak-fact}

Algebraic weak factorisation systems are an important tool for viewing
classes of maps commonly considered in homotopical algebra as
structure on a map, rather than a property of a map. In particular,
they play an important role in providing a structured version of Kan
fibration in cubical sets and simplicial sets \cite{gambinosattlerpi,
  swannomawfs, swanidams, awodey19}. Although they are usually defined
via functorial factorisations \cite{grandistholennwfs,
  garnersmallobject}, Bourke and Garner showed the definition is
equivalent to one based on double categories
\cite{bourkegarnerawfs1}. We give a mild reformulation of their
definition phrased in terms of notions of structure and some
definitions from the theory of comprehension categories with relevance
to the semantics of type theory.

We first note that notions of structure on codomain fibrations can be
seen as comprehension categories, as used in the semantics of type
theory \cite[Chapter 10]{jacobs}.

\begin{defn}[Jacobs]
  A \emph{comprehension category} is a Grothendieck fibration $p : \EE
  \to \BB$ together with a fibred functor $\chi$ from $p$ to the
  codomain fibration on $\BB$, as illustrated below.
  \begin{displaymath}
    \begin{tikzcd}
      \EE \ar[rr] \ar[dr, swap, "p"] & & \BB^\to \ar[dl, "\cod"] \\
      & \BB &
    \end{tikzcd}
  \end{displaymath}
\end{defn}

\begin{defn}
  We say a comprehension category is \emph{monadic} if $\chi$ is
  strictly monadic as a functor.
\end{defn}

\begin{rmk}
  By Proposition \ref{prop:monadcatcatover} we do not need to distinguish
  between \(\chi\) being monadic as a functor in \(\cats\) or
  \(\cats/\BB\). We do not require the monad to preserve cartesian
  maps.
\end{rmk}

As a special case of Lemma \ref{lem:algccl} we have:
\begin{prop}
  Any monadic comprehension category is a (necessarily monadic) notion
  of fibred structure on \(\cod : \BB^\to \to \BB\).
\end{prop}

Units are used in the theory of comprehension categories to model unit
types in type theory. We recall the strict version of the definition.
\begin{defn}
  A \emph{strict unit} is a functor $t : \BB \to \EE$ which is right
  adjoint to $p$ and such that $\chi(t(I)) = 1_I$ for all $I \in
  \BB$.
\end{defn}

\begin{prop}
  Every monadic comprehension category has a strict unit.
\end{prop}

\begin{proof}
  This follows from the fact that monadic functors create limits,
  noting that for each object $I$, $1_I$ is a terminal object in
  $\BB/I$.
\end{proof}

Suppose $\chi$ is a comprehension category with a strict unit
$t$. Then we have the following commutative diagram in $\cats$.
\begin{equation}
  \label{eq:4}
  \begin{tikzcd}[sep = 4.5em]
    \EE \ar[r, shift left=1.5ex, "\dom \circ \chi"]
    \ar[r, swap, shift right=1.5ex, "\cod \circ \chi"] \ar[d, "\chi"]
    & \BB \ar[d, equals] \ar[l, "t" description] \\
    \BB^\to \ar[r, shift left=1.5ex, "\dom"]
    \ar[r, swap, shift right=1.5ex, "\cod"] & \BB
    \ar[l, "1" description]
  \end{tikzcd}
\end{equation}
Note that the bottom row is an internal category in $\cats$, with
multiplication given by composition in $\BB$. Viewed as a double
category it is the double category of commutative squares in $\BB$.

\begin{defn}
  A \emph{composition functor} is a functor
  $\EE \times_{\BB} \EE \to \EE$ making the top row of \eqref{eq:4} an
  internal category, and the whole square a functor.
\end{defn}

We can think of composition functors as algebraic versions of
\(\Sigma\)-types in the following sense. In the theory of
comprehension categories we can implement \(\Sigma\)-types as an
operation \(\EE \times_\BB \EE \to \EE\) that commutes up to
isomorphism with composition in \(\BB\), referred to as \emph{strong
  coproducts} by Jacobs \cite[Definition 10.5.2]{jacobs}. Jacobs'
definition of strong coproducts further requires that the operation is
obtained from a dependent coproduct for the fibration \(p\). However,
we observe that we can satisfy this requirement by replacing the
morphisms of \(\EE\) with those in \(\BB\) to make \(\chi\) full and
faithful. Since the morphisms of \(\EE\) are not used in the
interpretation of type theory this has no effect therein. We can also
justify modifying Jacobs' definition in this way by considering the
construction of \(\Sigma\)-types in cubical sets
\cite{coquandcubicaltt}.  A dependent coproduct in a fibration,
\(\coprod_\sigma X\) is
uniquely determined up to isomorphism by the map \(\sigma : I \to J\)
and the object \(X\) in \(\EE_I\). However, \(\Sigma\)-types in
cubical sets are implemented by defining a Kan fibration structure on
the underlying \(\Sigma\)-type in the standard model of extensional
type theory in presheaves. The Kan fibration structure depends on the
fibration structures of \emph{both} types given as input. Hence we
should not expect it to be unique up to isomorphism of Kan fibration
structures if we are only given the map \(\sigma : I \to J\) without a
choice of fibration structure. It is however unique up to isomorphism
of underlying presheaves.

Composition functors are stronger than necessary to obtain \(\Sigma\)
types. In addition to Jacobs' strongness condition, they also satisfy
strict associativity as part of the definition of internal category,
which is not needed for type theory. However, it is natural to
consider \(\Sigma\)-types satisfying this additional requirement in
the setting of cofibrantly generated awfs's, where they occur
automatically.

For the semantics of type theory it is useful to observe that any
composition functor is automatically fibred, in the following sense.

\begin{prop}
  Any composition functor \(- \bullet - : \EE \times_\BB \EE \to \EE\)
  on a monadic comprehension category preserves cartesian maps in both
  arguments.
\end{prop}

\begin{proof}
  Suppose that \(f : X \to X'\) and \(g : Y \to Y'\) are composable
  and cartesian. Write \(\Gamma\) for \(p(Y')\) and \(\Delta\) for
  \(p(Y)\). Write \(\{ - \}\) for the composition \(\dom \circ
  \chi\). We then have the following commutative diagram in \(\BB\),
  where the upper commutative square is \(\chi(f)\), the lower
  commutative square is \(\chi(g)\), and the whole rectangle is the
  image under \(\chi\) of the composition \(g \bullet f\).
  \begin{displaymath}
    \begin{tikzcd}
      \{ X \} \ar[r] \ar[d] & \{X'\} \ar[d] \\
      \{ Y \} \ar[r] \ar[d] & \{ Y' \} \ar[d] \\
      \Delta \ar[r] & \Gamma
    \end{tikzcd}
  \end{displaymath}
  Since \(\chi\) preserves cartesian maps, the upper and lower squares
  are both pullbacks. Hence the big rectangle is a pullback. However,
  any monadic fibred functor reflects cartesian maps (since it
  reflects vertical isomorphisms), and so \(g
  \bullet f\) is cartesian, as required.
\end{proof}

\begin{thm}[Bourke-Garner]
  The above definition of awfs corresponds precisely to the more usual
  definition (appearing e.g. in \cite{garnersmallobject}, which aside
  from a distributive law condition is the same as given by Grandis
  and Tholen under the name natural weak factorisation system
  \cite{grandistholennwfs}).
\end{thm}

\begin{proof}
  This is a rephrasing of \cite[Proposition 4]{bourkegarnerawfs1}.
\end{proof}

We also give a fibred version of the definition of awfs, as in
\cite{swanliftprob}. Given a fibration \(p : \EE \to \BB\) note that
we can also view \(V(\EE)\) as a double category, and similarly to
before, we have the commutative diagram below.

\begin{equation}
  \label{eq:7}
  \begin{tikzcd}[sep = 4.5em]
    \FF \ar[r, shift left=1.5ex, "\dom \circ \chi"]
    \ar[r, swap, shift right=1.5ex, "\cod \circ \chi"] \ar[d, "\chi"]
    & \EE \ar[d, equals] \ar[l, "t" description] \\
    V(\EE) \ar[r, shift left=1.5ex, "\dom"]
    \ar[r, swap, shift right=1.5ex, "\cod"] & \EE
    \ar[l, "1" description]
  \end{tikzcd}
\end{equation}

\begin{defn}
  We say a \emph{fibred composition functor} on a notion of structure
  \(\chi : \FF \to V(\EE)\) is a functor
  \(\FF \times_\EE \FF \to \EE\) over \(\BB\) making the top row of
  \eqref{eq:7} an internal category and the whole diagram a double
  functor.
\end{defn}

\begin{defn}
  An \emph{algebraic weak factorisation system over a fibration
    \(p : \EE \to \BB\)} is a monadic notion of structure on
  \(V(\EE) \to \EE\) together with a fibred composition
  functor.
\end{defn}

\begin{defn}
  We say an algebraic weak factorisation system over
  \(p : \EE \to \BB\) is \emph{fibred} if it has a left adjoint that
  preserves the property of maps being cartesian over \(\BB\).

  We say it is \emph{strongly fibred} if it has a left adjoint that
  preserves the property of maps being cartesian over \(\EE\).
\end{defn}

\subsection{Lifting structures}
\label{sec:lifting-structures}

Let \(p : \EE \to \BB\) be a locally small bifibration and
fix a vertical map \(m : A \to B\) in \(\EE\).

\begin{defn}
  The \emph{lifting notion of structure generated by \(m\)} is the
  notion of structure on \(\cod : V(\EE) \to \EE\) defined as follows. An object of
  \(m^\pitchfork\) is a pair consisting of \(f : X \to Y \in V(\EE)\)
  together with a section of the canonical map
  \(\hom(B, X) \to \hom(A, X) \times_{\hom(A, Y)} \hom(B, Y)\), with
  the map \(\FF \to V(\EE)\) given by projection.
\end{defn}

\begin{prop}
  \label{prop:liftstructsigma}
  The lifting structure generated by \(m\) is a notion of
  structure on \(V(\EE) \to \EE\) and admits a fibred composition functor.
\end{prop}

It follows that if a lifting notion of structure of a map is monadic,
then it is automatically an awfs. We refer to awfs's of this form as
\emph{cofibrantly generated}.

\begin{rmk}
  By Proposition \ref{prop:monadcatcatover}, to show that
  \(m^\pitchfork \to V(\EE)\) is monadic over \(\BB\) it suffices to
  show it is monadic as a functor in \(\cats\).
\end{rmk}

\begin{thm}
  Suppose \(p : \EE \to \BB\) is complete and cocomplete (as a
  fibration) and \(m^\pitchfork \to V(\EE)\) is an awfs. Then it is
  a fibred awfs (i.e. its left adjoint is fibred).
\end{thm}

\begin{proof}
  See \cite[Theorem 5.5.1]{swanliftprob}.
\end{proof}

Algebraic versions of the small object argument can be seen as proofs
that certain lifting notions of structure are monadic.

\begin{thm}[Garner]
  Suppose that \(\EE \to \BB\) is a category indexed family fibration
  \(\fmly(\CC) \to \cats\) such that \(\CC\) is cocomplete and one of
  the following conditions holds.
  \begin{enumerate}
  \item For every \(X \in \CC\) there is a regular ordinal \(\alpha\)
    for which \(X\) is \(\alpha\)-presentable.
  \item \(\CC\) admits a proper well-copowered factorisation system
    \(\mathcal{E}, \mathcal{M}\) such that for every \(X \in \CC\)
    there is a regular ordinal \(\alpha\) for which \(X\) is
    \(\alpha\)-bounded with respect to \((\mathcal{E},
    \mathcal{M})\).
  \end{enumerate}
  Then for any family of maps \(m\), the lifting notion of structure
  generated by \(m\) is monadic.
\end{thm}

\begin{proof}
  See \cite{garnersmallobject}.
\end{proof}

\begin{thm}[Swan]
  Suppose that \(\EE \to \BB\) is a codomain fibration
  \(\BB^\to \to \BB\) on a locally cartesian closed category \(\BB\),
  that \(m\) is a family of maps and one of the following conditions
  holds.
  \begin{enumerate}
  \item \(\BB\) is locally cartesian closed, has exact quotients and
    \(W\)-types and satisfies \(\mathbf{WISC}\).
  \item \(\BB\) is an internal presheaf category in a locally
    cartesian closed category with finite colimits and disjoint
    coproducts and \(m\) is a pointwise decidable monomorphism.
  \end{enumerate}
  Then the lifting notion of structure generated by \(m\) is monadic.
\end{thm}

\begin{proof}
  See \cite{swanwtypered}.
\end{proof}

\begin{rmk}
  Regarding the connection between awfs's and the semantics of type
  theory, we observe that writing \(\chi\) for the forgetful functor
  \(m^\pitchfork \to V(\EE)\) the map \(\dom \circ \chi\) is right
  adjoint to the functor \(1_{(-)} \to m^\pitchfork\) sending each
  object of \(\EE\) to the terminal object of its fibre. It follows
  that for each \(I \in \BB\) the restriction of
  \(m^\pitchfork \to \EE_I\) admits comprehension in the sense of
  Example \ref{ex:ptdobjects}, and furthermore the
  resulting comprehension category with unit on \(\EE_I\) as in
  \cite[Definition 10.4.7]{jacobs} is the same as that given in
  section \ref{sec:algebr-weak-fact}.
\end{rmk}

\section{Other characterisations of definability}
\label{sec:other-char-local}

\subsection{Representable maps}
\label{sec:representable-maps}

\begin{thm}
  \label{thm:ccltosplitting}
  Suppose we are given a fibred functor $\chi : \DD \to \EE$ that
  creates cartesian lifts and a splitting on the fibration $\EE \to
  \BB$. Then we can define a splitting on $\DD$ that is strictly
  preserved by $\chi$.
\end{thm}

\begin{proof}
  Given an object $D$ of $\DD$ and a map $\sigma : I \to p(D)$, we
  have a choice of object $\sigma^\ast(\chi(D))$ and cartesian map
  $\bar{\sigma} : \sigma^\ast(\chi(D)) \to \chi(D)$ in $\EE$ over
  $\sigma$. We choose the splitting at $D$ to be the unique cartesian
  map over $\bar{\sigma}$ with codomain $D$.
\end{proof}

\begin{defn}
  \label{def:splittingtonattrans}
  When $\EE$ is split, we have a presheaf on $\BB$, by mapping
  $I \in \BB$ to the objects of $\EE_I$. We denote this presheaf
  $\tilde{\EE}$.

  By Theorem \ref{thm:ccltosplitting} we similarly have another
  presheaf sending $I$ to the set of objects of $\DD_I$, which we
  denote $\tilde{\DD}$, and we have a natural transformation
  $\tilde{\chi} : \tilde{\DD} \to \tilde{\EE}$.
\end{defn}

\begin{prop}
  When we are given a splitting of $\EE$, the natural transformation
  $\tilde{\chi}$ in definition \ref{def:splittingtonattrans} is a
  representable map in presheaves over $\BB$ if and only if $\chi$ is
  definable.
\end{prop}

As an alternative to requiring splitness, one can consider a
generalised definition of presheaf and representable map using
\(2\)-category theory that can be
obtained from any Grothendieck fibration.
This is one way of understanding the definition of local
representability given by Shulman \cite[Definition
3.10]{shulmaninftytopunivalence}.

\subsection{Pullbacks of the notion of structure of sections}
\label{sec:pullb-noti-struct}

This characterisation is based on Awodey's \emph{universal fibrations}
\cite[Section 6.3]{awodey19}. We can understand this definition as
follows. We saw in Lemma \ref{lem:sectlocrepr} that the notion of
structure on \(\cod : \BB^\to \to \BB\) of sections (Example
\ref{ex:sections}) is always definable. We will see below that it is
the universal example of definable notion of structure, in the sense
that every other definable notion of structure is a pullback of this
one. Hence we could alternatively define definable notions of
structures as fibred functors
\(\zeta : \cart(\EE) \to \cart(\BB^\to)\) such that
\(\cart(\DD) \to \cart(\EE)\) is the pullback of
\(\cart(\BB^\to_\bullet) \to \cart(\BB^\to)\) along \(\zeta\). This
was already observed by Shulman \cite[Proposition
2.7]{shulmaninftytopunivalence}, but for completeness we give a direct
proof in our formulation here.

\begin{thm}
  Let $\BB$ be a finitely complete category. A notion of fibred
  structure $\chi : \DD \to \EE$ over $\BB$ is definable
  if and only if there are fibred functors
  $\cart(\DD) \to \cart(\BB^\to_\bullet)$ and
  $\zeta : \cart(\EE) \to \cart(\BB^\to)$ making a (strict) pullback
  as illustrated below.
  \begin{displaymath}
    \begin{tikzcd}
      \cart(\DD) \ar[drr, phantom, very near start, "\lrcorner"]
      \ar[rr] \ar[d] & & \cart(\BB^\to_\bullet)
      \ar[d] \\
      \cart(\EE) \ar[rr] \ar[dr] & & \cart(\BB^\to) \ar[dl] \\
      & \BB &
    \end{tikzcd}
  \end{displaymath}
\end{thm}

\begin{proof}
  Suppose first that $\chi$ is definable. In this case we
  can assign for $X \in \EE$ a representing object for
  $\bar{\chi}_X$. We take $\zeta(X)$ to be the representing object in
  $\BB/q(X)$. Now given a cartesian map $f : X \to Y$, we know by
  Lemma \ref{lem:pblocrepr} that $q(f)^\ast(\zeta(Y))$ is a
  representing object for $\bar{\chi}_X$. This gives us a canonical
  isomorphism between $\zeta(X)$ and $q(f)^\ast(\zeta(Y))$ over
  $q(X)$. In turn this gives us a pullback square in $\BB$ with
  $\zeta(X)$ on the left and $\zeta(Y)$ on the right, i.e. a morphism
  in $\cart(\BB^\to)$. One can check that this construction preserves
  identities and composition giving a functor $\zeta$.

  Finally, we have for each $X$ a bijection between sections of the
  map $\zeta(X) \to q(X)$ and objects of $\chi^{-1}(X)$. One can check
  this is natural, giving us a pullback square.

  For the converse, we recall that $\BB^\to_\bullet \to \BB^\to$ is
  always definable by Lemma \ref{lem:sectlocrepr}. It
  follows that the same is true for
  $\cart(\BB^\to_\bullet) \to \cart(\BB^\to)$. But now using the
  pullback square we see that $\cart(\chi)$ is definable
  by Lemma \ref{lem:cclpbstab}, and so $\chi$ is too.
\end{proof}

\subsection{Small families of objects and universes}
\label{sec:universes}

The main motivation for Shulman introducing local representability in
\cite{shulmaninftytopunivalence} was to study universes in models of
type theory. In this section we recall, for reference, the relation
between definability and universes in fibrations.

\begin{defn}
  Let $q : \EE \to \BB$ be a Grothendieck fibration, and an object $V$
  of $\EE$, we say an object $X$ of $\EE$ is \emph{$V$-small} if there
  exists a cartesian map $X \to V$.
\end{defn}

Note that $V$ itself is $V$-small, since the identity map is
cartesian. Also note that if an object $U$ is $V$-small and another
object $X$ is $U$-small, then $X$ is also $V$-small.

We now consider a collection of objects indexed by a class
\(\mathcal{M}\), say \((V_\alpha)_{\alpha \in \mathcal{M}}\). Suppose
further that every object \(X\) is \(V_\alpha\)-small for some
\(\alpha \in \mathcal{M}\).\footnote{Moreover, assume we can choose a
  canonical such \(\alpha(X)\) for each \(X\).} For this general
definition it is technically possible to just take \(V_\alpha\) to be
the class of all objects. However, in practice we usually assume extra
conditions. For example, when working on a codomain fibration over a
locally cartesian closed category, we might require class of
\(V_\alpha\)-small maps to be closed under composition and dependent
products. We can satisfy this over presheaf categories by assuming
every set is contained in an inaccessible set, and taking \(V_\alpha\)
to be the Hofmann-Streicher universe on an inaccessible ordinal
\(\alpha\).

\begin{thm}
  \label{thm:locrepuniverses}
  A notion of structure \(\chi : \DD \to \EE\) is definable if and
  only if for all \(\alpha \in \mathcal{M}\) we can find a cartesian
  map \(i : U_\alpha \to V_\alpha\) such that for every
  \(V_\alpha\)-small object \(X\) witnessed by a cartesian map
  \(f : X \to V_\alpha\) there is a natural correspondence between
  structures on \(X\) and maps \(q(X) \to q(U_\alpha)\) factoring
  \(q(f)\) through \(q(i)\).
\end{thm}

\begin{proof}
  By Lemma \ref{lem:particalrightadjfib}, and the assumption that
  every object is \(V_\alpha\)-small for some \(\alpha\), the right
  adjoint to \(\cart(\chi)\) is defined on all objects if and only if
  it is defined on each \(V_\alpha\). By Lemma \ref{lem:partreprequiv}
  this is the same as each presheaf \(\bar{\chi}_{V_\alpha}\) being
  representable. However, expanding out the definition, this is
  precisely saying there is a natural correspondence between
  structures on \(X\) and maps \(q(X) \to q(U_\alpha)\) factoring
  \(q(f)\) through \(q(i)\).
\end{proof}

Note in particular that for presheaf categories we can obtain explicit
descriptions of the \(U_\alpha\) by Theorem
\ref{thm:locrepinpshbase}.

By ``truncating'' the above theorem we get the following corollary.

\begin{cor}
  If \(U_\alpha\) is as in the statement of Theorem
  \ref{thm:locrepuniverses} then an object \(X\) is \(U_\alpha\)-small
  if and only if it is \(V_\alpha\)-small and admits at least one
  structure.
\end{cor}

\section{Fibrewise definability}
\label{sec:fibr-defin}

There are different ways that a fibred awfs might be definable. We first
note that the most direct definition of definability automatically
holds in many situations:

\begin{thm}
  Let \(q : \EE \to \BB\) be a locally small fibration. Every fibred
  awfs over \(q\), regarded as a notion of structure on
  \(V(\EE) \to \BB\), is definable.
\end{thm}

\begin{proof}
  This is a special case of Theorem \ref{thm:fibmonadlocrepr}.
\end{proof}

Often this kind of definability is automatically true for cofibrantly
generated awfs's:
\begin{thm}
  Let \(q : \EE \to \BB\) be a locally small fibration and suppose
  \(\BB\) is locally cartesian closed. Every lifting notion of
  structure on \(q\) is definable.
\end{thm}

\newcommand{\sq}{\operatorname{Sq}}

\begin{proof}
  Write \(h\) for the composition of maps
  \begin{equation*}
    \hom(A, X) \times_{\hom(A, Y)} \hom(B, Y) \longrightarrow I \times
    J
    \longrightarrow J
  \end{equation*}
  and \(p\) for the canonical map \(\hom(B, Z) \to \hom(m, f)\).

  For any \(\sigma : K \to J\), we have the following commutative
  diagram.
  \begin{displaymath}
    \begin{tikzcd}
      \hom(B, \sigma^\ast(Y)) \ar[d, "p'"] \ar[r, "\cong"] &
      \sigma^\ast(\hom(B, Y)) \pbcorner \ar[r] \ar[d, "\sigma^\ast(p)"] &
      \hom(B, Y) \ar[d] \\
      \hom(m, \sigma^\ast(f)) \ar[r, "\cong"] &
      \sigma^\ast(\hom(m, f)) \ar[r] \ar[ur, dotted] &
      \hom(m, f)
    \end{tikzcd}
  \end{displaymath}
  Lifting structures on \(\sigma^\ast(f)\)
  correspond precisely to sections of \(p'\), which
  correspond precisely to maps
  \(\sigma^\ast(\hom(m, f)) \to \hom(B, Y)\) making commutative
  triangles, as in the dotted diagonal arrow above. However, such maps
  correspond precisely to commutative triangles of the form below.
  \begin{displaymath}
    \begin{tikzcd}
      K \ar[rr, dotted] \ar[dr, "\sigma"] & & \prod_h(\hom(B, Z))
      \ar[dl, "\prod_h p"] \\
      & J &
    \end{tikzcd}
  \end{displaymath}
  However, this verifies that \(\prod_h p\) is indeed the representing
  object required to show that lifting structures are definable.
\end{proof}

However, in practice for the semantics of type theory we are not so
much interested in the entire fibred awfs, but only the ordinary awfs
given by restriction to the terminal fibre \(\EE_1\). This awfs is not
necessarily definable, even if it is the restriction of a fibred awfs
that is definable. Hence in this paper we mainly consider the
following stronger notion, that we denote fibrewise definability.

\begin{defn}
  Suppose we are given a fibration \(q : \EE \to \BB\). A notion of
  structure on the fibration \(\cod : V(\EE) \to \EE\) is
  \emph{fibrewise definable} if for each \(I \in \BB\) the notion of
  structure on \(\EE_I^\to \to \EE_I\) given by restricting to the fibre category \(\EE_I\) is definable.
\end{defn}

This version of definability holds automatically for strongly fibred
awfs's, as shown below, but we will see some natural examples of
fibred awfs's where it does not.
\begin{thm}
  Let \(q : \EE \to \BB\) be a locally small fibration. Suppose that
  each fibre category \(\EE_I\) has dependent products. Then every
  strongly fibred awfs on \(q\) is fibrewise definable.
\end{thm}

\begin{proof}
  We again use Theorem \ref{thm:fibmonadlocrepr}.
\end{proof}

\section{Tiny Objects and Definable Awfs's}
\label{sec:tiny-objects-locally}

Throughout this section we will assume that the Grothendieck fibration
$q : \EE \to \BB$ is locally small and that $\BB$ has all finite
limits. We write $\hom(A, X)$ for the hom objects, following the
convention that $A$ and $X$ can lie in different fibres. If $A$ is in
the fibre over $I$ and $X$ in the fibre over $J$, then $\hom(A, X)$ is
isomorphic to $\hom_{I \times J}(\pi_0^\ast(A), \pi_1^\ast(X))$, where
$\pi_i$ are the projection maps out of $I \times J$. In particular,
when $J$ is the terminal object, $\hom(A, X)$ is the same as
$\hom_I(A, I^\ast(X))$.

We will give a sufficient criterion for a lifting notion of structure
to be fibrewise definable, based on a definition of \emph{family of
  tiny objects}. The argument is based on existing
constructions in cubical sets \cite{lops, awodey19}, but is more
general in three respects:
\begin{enumerate}
\item By working in an arbitrary Grothendieck fibration we can not
  only use objects that are tiny in the internal sense, that
  exponentiation has a right adjoint, but by applying the the result
  to set indexed families, we can also use objects that are tiny in
  the external sense, that their hom set functor has a right adjoint.
\item Instead of focusing on the particular definition of Kan
  fibrations, we consider cofibrantly generated fibred awfs's more
  generally. It turns out that for our sufficient criterion to apply,
  only the codomain of the generating family of left maps matters.
\item We consider not just individual tiny objects in a category, but
  \emph{families} of tiny objects. Whereas the definition of Kan
  fibration only features one tiny object, the interval, we will also
  see in Example \ref{ex:natmodelex} a definable awfs using the fact
  that the ``family of all identity types'' in a natural model can be
  seen as a family of tiny objects.
\end{enumerate}

\subsection{Tiny families of objects}
\label{sec:tiny-famil-objects}

\begin{defn}
  Let $I$ and $J$ be a elements of $\BB$ and $B$ an object of
  $\EE_I$. We say $B$ is \emph{tiny relative to $J$} if the functor
  $\EE_J \to \BB / (I \times J)$ defined as $\hom(B, -)$ has a right
  adjoint. We say it is a \emph{tiny family of objects} if it is tiny
  relative to \(J\) for all \(J\).
\end{defn}

\begin{example}
  Suppose that $\EE \to \BB$ is the fibration of set indexed families
  on a category $\mathbb{C}$. Then a family $(B_i)_{i \in I}$ is tiny
  relative to $J$ (for any $J$) if and only if each object $B_i$ is
  externally tiny, i.e. the hom set functor \(\hom(B_i, -)\) has a
  right adjoint. If $\mathcal{C}$ is a presheaf category then the
  tiny objects are precisely retracts of representables.
\end{example}

\begin{example}
  Suppose that $\EE \to \BB$ is the fibration of category indexed
  families on a category $\mathbb{C}$. Then a diagram
  $D : \mathcal{I} \to \mathbb{C}$ is tiny relative to $\mathcal{J}$
  when the functor
  $(D \downarrow -) : \mathbb{C}^\mathcal{J} \to \cats/(\mathcal{I}
  \times \mathcal{J})$ has a right adjoint.
\end{example}

We can give the following sufficient criterion for diagrams of
presheaves to be tiny over category indexed families.
\begin{lemma}
  Suppose that $\EE \to \BB$ is the fibration of category indexed
  families on a presheaf category $\psh{C}$ and we are given a diagram
  $D : \mathcal{I} \to \mathcal{C}$. Then $\yoneda \circ D$ is tiny
  relative to $\mathcal{J}$ for any small category $\mathcal{J}$.
\end{lemma}

\begin{proof}
  Given a representable $\yoneda c$ in $\psh{C}$, we see by the Yoneda
  lemma that
  $(\yoneda \circ D \downarrow \yoneda c) \cong (D \downarrow
  c)$. Hence if a right adjoint $G : \cats/\mathcal{I} \to \psh{C}$
  exists at all, we must have that
  $G(\mathcal{J} \to \mathcal{I})(c)$ is the set of
  functors $(D \downarrow c) \to \mathcal{J}$ over $\mathcal{I}$, with
  the action on morphisms given by composition. One can check that
  this does indeed give a right adjoint.
\end{proof}

\begin{example}
  \label{ex:codtiny}
  If $\BB^\to \to \BB$ is a codomain fibration, and $B$ is tiny in the
  internal sense i.e. $(-)^B$ has a right adjoint, then the family
  $B \to 1$ is tiny relative to $1$.
\end{example}

Example \ref{ex:codtiny} is most useful when combined with the results
below.

\begin{lemma}
  Suppose that $\BB$ is locally cartesian closed. Suppose that
  $B \in \EE$ is tiny relative to $J$ and we are given a cartesian map
  $B' \to B$ over $\sigma : I \to K$, say. Then $B'$ is also tiny
  relative to $J$.
\end{lemma}

\begin{proof}
  For each $X \in \EE_J$ we have the canonical pullback square
  below.
  \begin{equation*}
    \begin{tikzcd}
      \hom(B', X) \ar[r] \ar[d] \ar[dr, phantom, very near start,
      "\lrcorner"] & \hom(B, X) \ar[d] \\
      I \ar[r, "\sigma"] & J
    \end{tikzcd}
  \end{equation*}
  Hence we can factor the functor $\hom(B', -)$ as
  $\hom(B, -) : \mathbb{E}_1 \to \mathbb{B}/J$ followed by the functor
  $\sigma^\ast : \mathbb{B}/J \to \mathbb{B}/I$ given by pullback. The
  former has a right adjoint by the assumption that $B$ is tiny, and
  the latter by the assumption that $\mathbb{B}$ is locally cartesian
  closed.
\end{proof}

\begin{thm}[Freyd--Yetter]
  Suppose that an object $X \to 1$ is tiny relative to $1$ over a
  codomain fibration on a locally cartesian closed category with a
  classifier for regular monomorphisms.\footnote{Such categories are
    sometimes referred to as \emph{quasitoposes}.} Suppose further
  that the map $X \to 1$ is a regular epimorphism. Then $X$ is tiny
  relative to $J$ for any object $J$.
\end{thm}

\begin{proof}
  Essentially this is \cite[Theorem 1.4]{yetter} aside from some
  rephrasing and the observations that the regular epimorphism
  condition is necessary\footnote{This was later noted by Yetter in an
    erratum to \cite{yetter}.} and that a classifier for regular
  monomorphisms suffices for the proof in place of a subobject
  classifier.
\end{proof}

It is well known that if the product functor \( - \times c\) is
defined, then the representable \(\yoneda(c)\) is internally
tiny. This was generalised to tiny families by Newstead
\cite[Section 3.3]{newsteadthesis}.

\begin{thm}[Newstead]
  \label{thm:reprtotiny}
  We work over the codomain fibration on a presheaf category
  \(\psh{C}\). If \(f : X \to I\) is representable as a map of
  presheaves, then it is tiny relative to \(J\) for all \(J\).
\end{thm}

\begin{proof}
  Note that if \(f\) is representable, then we have a functor
  \(F : \int_{\mathcal{C}} I \to \mathcal{C}\) such that for all
  \((c, i) \in \int_{\mathcal{C}} I\) we have a pullback diagram of
  the form below:
  \begin{displaymath}
    \begin{tikzcd}
      \yoneda (F(c, i)) \pbcorner \ar[r] \ar[d] & X \ar[d] \\
      \yoneda c \ar[r] & I
    \end{tikzcd}
  \end{displaymath}

  For each \(J\), we can show that \(\hom(X, -)\) is isomorphic to a
  functor obtained by reindexing along the functor
  \(F_J : \int_\mathcal{C} I \times J \to \int_C J\) defined by
  \(F_J(c, i, j) := (F(c, i), j)\): for any \(c \in \mathcal{C}\) maps
  \(\yoneda c \to \hom_{I \times J}(J^\ast(X), I^\ast(Y))\) correspond,
  by the the adjunction between local exponentials and pullback, to
  maps \(\yoneda c \times_I X \to Y\), which correspond precisely to
  elements of \(Y(F(c, i))\) in the fibre of \(j\).
  \begin{displaymath}
    \begin{tikzcd}[column sep=8em]
      \psh{C}/J \ar[r, "{\hom_{I \times J}(J^\ast(X), I^\ast(-))}"]
      \ar[d, "\simeq"] \ar[dr, phantom, "\cong"]
      & \psh{C}/(I \times J) \ar[d,
      "\simeq"] \\
      \psh{\int_{{\mathcal{C}}} J} \ar[r, "F_J^\ast"] &
      \psh{\int_{\mathcal{C}} I \times J}
    \end{tikzcd}
  \end{displaymath}
  However, \(F_J^\ast\) has a right adjoint given by right Kan
  extension, so we are done.
\end{proof}

\begin{rmk}
  Newstead also showed a converse statement when the small category
  \(\mathcal{C}\) is Cauchy complete and has finite products. In this
  case every tiny family of objects is a representable map.
\end{rmk}

\begin{cor}
  Suppose that for an object \(c\) of a small category
  \(\mathcal{C}\), the product functor \(- \times c\) exists. Then
  the representable object \(\yoneda c\) is tiny as an object in
  presheaves.
\end{cor}

\begin{proof}
  If \(- \times c\) exists then the unique map \(\yoneda c \to
  1\) is representable.
\end{proof}

\subsection{Definability from Tiny Codomain}
\label{sec:defin-from-tiny}

We now use tininess to give examples of fibrewise definable lifting
notions of structure.

We will use the following observation.
\begin{lemma}
  \label{lem:pbulpsoln}
  Suppose we are given vertical maps $m : A \to B$, $f : X \to Y$ and
  $g : Z \to Y$ (where $f$ and $g$ necessarily lie in same the
  fibre). Write $g^\ast(f)$ for the pullback of $f$ along $g$. Then
  solutions to the universal lifting problem from $m$ to $g^\ast(f)$
  correspond precisely to maps $\hom(A, X) \times_{\hom(A, Y)} \hom(B,
  Z) \to \hom(B, Y)$ making a commutative triangle as below.
  \begin{equation}
    \label{eq:3}
    \begin{tikzcd}
      \hom(A, X) \times_{\hom(A, Y)} \hom(B, Z) \ar[r] \ar[d, swap,
      "{\langle 1_{\hom(A, X)} , \hom(B, g) \rangle}"] &
      \hom(B, X) \ar[dl] \\
      \hom(A, X) \times_{\hom(A, Y)} \hom(B, Y) &
    \end{tikzcd}
  \end{equation}
\end{lemma}

\begin{proof}
  In the fibre category over $\hom(A, X) \times_{\hom(A, Y)} \hom(B,
  Z)$ we can construct a commutative diagram of the following form.
  \begin{equation}
    \label{eq:1}
    \begin{tikzcd}
      \sigma^\ast(A) \ar[rr] \ar[d] & & \tau^\ast(X) \ar[d] \\
      \sigma^\ast(B) \ar[r] & \tau^\ast(Z) \ar[r] & \tau^\ast(Y)
    \end{tikzcd}
  \end{equation}

  By the universal property of the pullback, this factors as two
  squares, below.
  \begin{equation}
    \label{eq:2}
    \begin{tikzcd}
      \sigma^\ast(A) \ar[r] \ar[d] & \tau^\ast(g^\ast(X)) \ar[r]
      \ar[d] \ar[dr, phantom, very near start, "\lrcorner"]      
      & \tau^\ast(X) \ar[d] \\
      \sigma^\ast(B) \ar[r] & \tau^\ast(Z) \ar[r] & \tau^\ast(Y)      
    \end{tikzcd}
  \end{equation}

  One can check, for example by directly verifying the relevant
  universal property, that the left hand square is exactly the
  universal lifting problem from $m$ to $g^\ast(f)$.

  Again applying the universal property of the pullback, diagonal
  fillers in the left hand square of \eqref{eq:2} correspond precisely
  to diagonal fillers of \eqref{eq:1}. Maps
  $\sigma^\ast(B) \to \tau^\ast(X)$ correspond precisely to maps
  $\hom(A, X) \times_{\hom(A, Y)} \hom(B, Z) \to \hom(B, X)$ by the
  universal property of $\hom(B, X)$, and the upper and lower
  triangles commute for the diagonal filler if and only if the
  triangle in \eqref{eq:3} commutes.
\end{proof}

\begin{thm}
  \label{thm:tinytolocrepr}
  Suppose that $\BB$ is locally cartesian closed. Suppose we are given
  a vertical map $m : A \to B$ in $\EE_I$ where $B$ is tiny relative
  to \(J \in \BB\). Then the restriction of the lifting notion of
  structure generated by \(m\) to \(J\) is definable.
\end{thm}

\begin{proof}
  Let $G : \BB / (I \times J) \to \EE_J$ be the right adjoint to $\hom(B,
  -)$.
  
  Given $f : X \to Y \in \EE_J$, we will show that the presheaf
  $\bar{\chi}_{\cod(f)}$
  from Definition \ref{def:locpresheaf} is representable.

  Write
  $p : \hom(m, f) \to \hom(B,
  Y)$ for the projection map, and the canonical map
  $\hom(B, X) \to \hom(m, f)$
  as $t$. We construct the dependent product
  $\prod_p t : \prod_{p} \hom(B, X) \to \hom(B,
  Y)$. Viewing this as a map in $\BB/(I \times J)$, we apply $G$ to get a
  map $G(\prod_{p} \hom(B, X)) \to G(\hom(B, Y))$. We
  pullback along the unit map $\eta_Y : Y \to G(\hom(B, Y))$ to get a
  map $Y \times_{G(\hom(B, Y))} G(\prod_{p} \hom(B, X)) \to Y$. We
  will show this is representing for the presheaf $\bar{\chi}_{q(f)}$.

  Fix a map $g : Z \to Y$. Maps from $g$ to $f$ in $\EE_J/Y$
  correspond naturally by the universal property of the pullback to
  maps $h : Z \to G(\prod_{p} \hom(B, X))$ forming a commutative
  square as below.
  \begin{displaymath}
    \begin{tikzcd}
      Z \ar[r, "h"] \ar[d, "g"] & G(\prod_{p} \hom(B, X)) \ar[d] \\
      Y \ar[r, "\eta_Y"] & G(\hom(B, Y))
    \end{tikzcd}
  \end{displaymath}

  Passing across the adjunction, $\hom(B, -) \dashv G$, we see that
  such maps correspond to the maps $\hom(B, Z) \to \prod_p \hom(B, X)$
  in the commutative square below.
  \begin{displaymath}
    \begin{tikzcd}
      \hom(B, Z) \ar[r] \ar[d, "{\hom(B, g)}"] & \prod_{p} \hom(B, X) \ar[d] \\
      \hom(B, Y) \ar[r, equals] & \hom(B, Y)
    \end{tikzcd}
  \end{displaymath}

  Rearranging, passing across the pullback-dependent product
  adjunction and simplifying allows us to apply Lemma
  \ref{lem:pbulpsoln} to show such diagrams correspond precisely to
  solutions of the universal lifting problem of $m$ against
  $g^\ast(f)$.
\end{proof}

\begin{cor}
  \label{cor:setreprcodtolocrepr}
  Let $\mathcal{C}$ be a small category, and
  $(m_i : A_i \to \yoneda B_i)$ a family of maps in the presheaf
  category $\psh{C}$ with representable codomain. Then the awfs
  cofibrantly generated by the family is definable.
\end{cor}

\begin{cor}
  For any small category $\mathcal{C}$ with a wfs $(\mathcal{L},
  \mathcal{R})$, there is a definable awfs on $\psh{C}$
  such that the Yoneda embedding preserves and reflects left maps and
  right maps.
\end{cor}

\begin{proof}
  Take the generating left maps to be the image of the left maps in
  $\mathcal{C}$ under the Yoneda embedding.
\end{proof}

\begin{cor}
  Let $\BB$ be a locally cartesian closed category and let $B$ be an
  internally tiny object in $\BB$. For any object $I$ and any map $m$
  of the form $A \to I^\ast(B)$ in the slice category $\BB/I$, the
  awfs cofibrantly generated by $m$ is definable.
\end{cor}

\begin{cor}
  \label{cor:catindexedtinytolocrepr}
  Let $\EE \to \BB$ be the category indexed families
  fibration for a category $\mathbb{C}$ and let $M : \mathcal{I} \to
  \mathbb{C}^\to$ be a diagram of left maps that cofibrantly generates
  an awfs $(L, R)$. If $\cod \circ M$ is tiny, then the awfs is
  definable.
\end{cor}

\begin{proof}
  Since $\cats$ is not locally cartesian closed, we need to verify
  that the particular dependent product used in the proof of Theorem
  \ref{thm:tinytolocrepr} exists, so that we can apply the same proof
  as before.

  Write $A$ for $\cod \circ M$ and $B$ for $\dom \circ M$.

  Explicitly, for each morphism $f : X \to Y$ in $\mathbb{C}$, we need
  to construct a dependent product along the canonical functor
  $p : (A \downarrow X) \times_{(A \downarrow Y)} (B \downarrow Y) \to
  (B \downarrow Y)$. However, this functor is a discrete fibration,
  since it is a pullback of the discrete fibration
  $(A \downarrow X) \to (A \downarrow Y)$. Hence the dependent product
  along $p$ exists, and so we can continue following the same proof as
  in Theorem \ref{thm:tinytolocrepr}.
\end{proof}

\begin{cor}
  For any small category $\mathcal{C}$ with an awfs $(L, R)$, there is
  a definable awfs on $\psh{C}$ such that the Yoneda
  embedding lifts to functors from the categories of (co)algebra
  structures in $\mathcal{C}$ to those in $\psh{C}$.
\end{cor}

\begin{proof}
  We take the generating diagram of left maps to be the composition
  $L\text{-}\operatorname{Map} \to \mathcal{C}^\to \to
  {\psh{C}}^{\to}$. It is clear by definition that composition with
  the codomain map factors through the Yoneda embedding.
\end{proof}

\newcommand{\UU}{\mathcal{U}}

\begin{example}
  \label{ex:natmodelex}
  In Awodey's natural models \cite[Section 2.4]{awodeynatmod},
  intensional identity types are implemented as maps \(\mathsf{Id} :
  \tilde{\mathcal{\UU}} \times_\mathcal{\UU} \tilde{\mathcal{\UU}}\) and
  \(\mathsf{i}\) making a commutative square as below:
  \begin{displaymath}
    \begin{tikzcd}
      \tilde{\UU} \ar[d, "\delta"]
      \ar[r, "\mathsf{i}"] & \tilde{\UU} \ar[d, "p"] \\
      \tilde{\UU} \times_\UU \tilde{\UU} \ar[r, "\mathsf{Id}"] & \UU
    \end{tikzcd}
  \end{displaymath}
  We can view the map \(\rho\), given by the universal property of the
  pullback below, as the ``universal reflexivity map.''
  \begin{displaymath}
    \begin{tikzcd}      
      \tilde{\UU} \ar[ddr, swap, bend right, "\delta"] \ar[drr, bend left,
      "\mathsf{i}"] \ar[dr, dotted, "\rho"] & & \\
      & I \pbcorner \ar[r] \ar[d] & \tilde{\UU} \ar[d, "p"] \\
      & \tilde{\UU} \times_\UU \tilde{\UU} \ar[r, "\mathsf{Id}"] & \UU
    \end{tikzcd}
  \end{displaymath}
  Viewing \(\rho\) as a family of maps over
  \(\tilde{\UU} \times_\UU \tilde{\UU}\) in the codomain fibration,
  lifting structures against \(\rho\) are used in loc. cit. to
  implement the \(\mathsf{j}\) terms in type theory. Note that the
  pullback projection map
  \(I \to \tilde{\UU} \times_{\UU} \tilde{\UU}\) is a pullback of a
  representable map of presheaves, and so representable itself. It
  follows by Theorems ~\ref{thm:tinytolocrepr} and
  ~\ref{thm:reprtotiny} that the lifting notion of structure generated
  by \(\rho\) is definable. Also note that any of the small object
  arguments listed in Section ~\ref{sec:lifting-structures} can be
  used to show the lifting notion of structure is monadic.
\end{example}

\section{Non definable examples}
\label{sec:non-loc-repr}

\subsection{A review of Kan fibrations and Hurewicz Fibrations}
\label{sec:revi-kan-fibr}

Kan fibrations are one of the key ingredients to the standard model
structure on simplicial sets \cite{quillen67}. In this section we give
a general definition of Kan fibration over a Grothendieck fibration,
\(q : \EE \to \BB\). The definition is not the most general
possible\footnote{Two possible generalisations are to replace an
  interval object with two global endpoints with a single generic
  point, as in \cite{awodey19}, and to replace cartesian product with
  a general fibred monoidal product.} but is enough to cover most of
the cases we will consider in this paper. For convenience we will
assume that the base \(\BB\) has all finite limits and colimits, and
that \(q\) has fibred products. We will write \(\hat{\times}\) for
the pushout product on cartesian product. See e.g. \cite[Section
11.1]{riehlcht} for a standard reference on pushout product.

\begin{defn}
  Let \(q : \EE \to \BB\) be a locally small fibration. Suppose we are
  given an interval object \(1 \rightrightarrows \II\) in \(\EE_1\)
  and a vertical monomorphism \(m : A \to B\) over \(I \in \EE_I\).

  We say a vertical map \(f : X \to Y\) is a \emph{Kan fibration} if
  it has the fibred right lifting property against the following
  family of maps: we first form the pushout products
  \(\delta_i \hat{\times}_I m\) in \(\EE_I\) for \(i = 0, 1\), and
  then take their coproduct to obtain a vertical map over \(I + I\).
\end{defn}

\begin{example}
  \label{ex:gzb2}
  We work over the set indexed family fibration on simplicial sets. We
  take the interval object to be \(\Delta_1\) and \(m\) to be the set
  indexed family of all boundary inclusions
  \(\partial \Delta_n \hookrightarrow \Delta_n\). In this way we
  obtain \cite[Chapter IV, Section 2, \(B_2\)]{gabrielzisman}.
\end{example}

\begin{example}
  \label{ex:gzb3}
  We work over the codomain fibration on simplicial sets. We take the
  interval object again to be \(\Delta_1\). We take \(m\) to be the
  subobject classifier \(1 \to \Omega\) viewed as an object in
  \(\sset/\Omega\). This gives us \cite[Chapter IV, Section 2,
  \(B_3\)]{gabrielzisman}.
\end{example}

\begin{example}
  As a generalisation of Example \ref{ex:gzb3} we can work over an
  arbitrary topos with connected interval with disjoint endpoints,
  replace the subobject classifier with a classifier for a subclass of
  monomorphisms closed under composition and finite unions, and
  containing the endpoints of the interval. This gives the definition
  of Kan fibration in \cite{pittsortoncubtopos}\footnote{for Kan
    filling, rather than composition}.
\end{example}

We also consider the following degenerate example:
\begin{defn}
  A Hurewicz fibration is a Kan fibration where \(m\) is the unique
  map in \(\EE_1\) from the initial object to the terminal object.
\end{defn}

\subsection{Full notions of structure and the axiom of choice}
\label{sec:some-non-definable}

Our first examples of non-definable notions of structure will be full
notions of structure arising from certain weak factorisation
systems. The intuitive idea behind these results is that from any
notion of structure we can obtain a full notion of structure by image
factorisation, i.e. we can consider the class of objects ``admitting
at least one structure.'' We can give a general rule of thumb that in
the absence of the axiom of choice this is often unreasonable and can
lead to non definable full notions of structure, even when the
original (non-full) notion of structure is definable. We first
illustrate this idea with the very simple example of split
epimorphisms. As we saw in Example \ref{ex:sections} sections form a
definable notion of structure on a codomain fibration.

\begin{thm}
  Suppose that the class of split epimorphisms in $\BB$ is definable
  as a full notion of structure on $\cod : \BB^\to \to
  \BB$. Then every regular epimorphism splits.
\end{thm}

\begin{proof}
  Let $f : X \to Y$ be a map in $\BB$ and let
  $m : I \rightarrowtail Y$ be the representing object at $f$. Note
  that if we pull $f$ back along itself, then the projection map
  $X \times_Y X \to X$ is a split epimorphism, with the diagonal map
  $\Delta : X \to X \times_Y X$ as section. Hence $f$ factors through
  $m$. Now if $f$ is a regular epimorphism, then it is left orthogonal
  to any monomorphism, giving us a section of $m$. Hence $f$ is a
  split epimorphism.  
\end{proof}

We now show the same basic idea applies when we consider the
underlying wfs of a wide range of awfs's in presheaf categories.

\begin{thm}
  \label{thm:wfsnotlocrepr}
  Suppose we are given a weak factorisation system on a presheaf
  category \(\psh{C}\) satisfying the following conditions:
  \begin{enumerate}
  \item The wfs is generated by locally decidable monomorphisms.
  \item \label{it:propermono} There is a left map \(m : A \to B\) such
    that there is \(c \in \mathcal{C}\) and \(x \in B(c)\) which does
    not lie in the image of \(m_c\).
  \end{enumerate}

  If the notion of structure given by right maps (as in Example
  \ref{ex:wfsfns}) is definable then the axiom of choice holds.
\end{thm}

\begin{proof}
  Suppose we are given a family of merely inhabited sets
  \((X_i)_{i \in I}\). We will construct a choice function for the
  family \((X_i)_{i \in I}\) from the assumption that the wfs is
  definable.

  Let \(m : A \to B\) be a left map satisfying condition
  \ref{it:propermono}. Write \(\Omega_{\decble} \) for the classifying
  object for locally decidable subobjects. For each \(i \in I\) we
  consider the factorisation \(m\) given by the awfs cofibrantly by
  the following family of objects over the codomain fibration on
  \(\psh{C}\).
  \begin{displaymath}
    \begin{tikzcd}
      X_i \ar[rr, "p_i"] \ar[dr] & & \Omega_{\decble} \times X_i
      \ar[dl] \\
      & \Omega_{\decble} \times X_i &
    \end{tikzcd}
  \end{displaymath}
  Intuitively the factorisation freely adds a filler for each lifting
  problem from a locally decidable monomorphism to \(m\) \emph{and for
    each \(x \in X_i\)}. Hence we can find a filler for each lifting
  problem, given a choice of \(x \in X_i\).
  
  We will write this factorisation as \((L_i, R_i)\), so we are
  considering the map \(f_i := R_i m : K_i m \to B \).

  Given a fixed \(x \in X_i\) we can choose for each locally decidable
  monomorphism \(n : C \to D\) a map \(c : D \to \Omega_{\decble}\)
  such that \(n\) is the pullback of \(p_i\) along \(c\). Hence we can
  assign each left map an \(L_i\)-coalgebra structure and thereby a
  choice of diagonal filler for each lifting problem of a diagonal map
  against \(f_i\). Since \(X_i\) is merely inhabited it follows that
  there exists a function witnessing that \(f_i\) has the right
  lifting property against each left map, and so is a right map. By
  the assumption of definability, it follows that the coproduct
  \(\coprod_{i \in I} f_i : \coprod_{i \in I} K_i m \to I \times B\)
  is also a right map. Hence there is a function assigning a choice of
  filler for each lifting problem against \(m\). In particular, for
  each \(i \in I\) we have a choice of map \(j_i\) for each of the
  following lifting problems.
  \begin{displaymath}
    \begin{tikzcd}
      A \ar[d, "m"] \ar[r, "{\langle i, L_i \rangle}"] &
      \coprod_{i \in I} K_i m \ar[d] \\
      B \ar[r, swap, "{\langle i, 1_b \rangle}"] \ar[ur, dotted, "j_i"
      description] & I \times B
    \end{tikzcd}
  \end{displaymath}
  However, we can now read off from the explicit construction of
  \(K_i m\) in presheaves \cite{swanwtypered} that each element of
  \(K_i m(c)\) is either in the image of \(L_i m\) or of the form
  \(\sup(z, \alpha)\) where \(z\) belongs to
  \(\Omega_{\decble} \times X_i(c)\) and \(\alpha\) is a dependent
  function to earlier constructed elements. We choose the object \(c\)
  of \(\mathcal{C}\) as in the condition on \(m\) in the statement of
  the theorem to ensure the former case is not possible
  leaving only the latter case. In particular \(\pi_1(z)\) belongs to
  \(X_i\), giving us a choice function for the family \((X_i)_{i \in I}\).
\end{proof}

We can use Theorem \ref{thm:wfsnotlocrepr} to give a concrete example
of a non definable weak factorisation system in simplicial
presheaves. Given a small category \(\mathcal{D}\), the category of
\emph{simplicial presheaves} is by definition the category of
presheaves on \(\mathcal{D} \times \simpcat\). We can view this as the
category of category of simplicial sets constructed internally in the
presheaf topos \(\psh{D}\). Following the work of Gambino, Henry,
Sattler, Szumi\l{}o
\cite{gambinosattlerszumilo, henryexinfty, ghsseffectivemodelstr} we can
construct a model structure on simplicial presheaves using the
internal logic of \(\psh{D}\). They define Kan fibrations as
cofibrantly generated by the set of horn inclusions. We can read off
an external description of the Kan fibrations as follows. We can
equivalently view the category of simplicial presheaves as the
category of functors \([\mathcal{D}, \sset ]\). We can then read off
the internal definition of Kan fibration as being the same as that
given by the \emph{pointwise awfs} \cite[Section 4.2]{riehlams} on Kan
fibrations. Explicitly a natural transformation \(f\) between functors
\([\mathcal{D}, \sset ]\) is a right map if we can assign Kan
fibration structures \(f_d\) for each object \(d\) of \(\mathcal{D}\)
in such as way that for each morphism \(\sigma : d \to d'\) in
\(\mathcal{D}\) is a morphism of Kan fibrations. That is, the diagonal
fillers are chosen so that the triangles in the centre of each diagram
below commute.
\begin{displaymath}
  \begin{tikzcd}
    \Lambda^i_n \ar[r] \ar[d] & X_{d'} \ar[r, "X_\sigma"] \ar[d] &
    X_{d} \ar[d, "f_{d'}"] \\
    \Delta_n \ar[r] \ar[ur, dotted] \ar[urr, dotted, crossing over] &
    Y_{d'} \ar[r, "Y_\sigma"] & Y_{d}
  \end{tikzcd}
\end{displaymath}

We note that this definition gives a definable awfs. This follows from
our general result, but already appears implicitly in the construction
of the universe by Gambino and Henry \cite{gambinohenry}. However, it
is commonly the case for the axiom of choice to fail in presheaf
toposes, and since horn inclusions are locally decidable we can apply
Theorem \ref{thm:wfsnotlocrepr} to show that the full notion of
structure from the wfs underlying the awfs is not definable. For
instance, this applies for the very simple example of simplicial
presheaves on the walking arrow \(\cdot \to \cdot\), since
\(\sets^\to\) is not boolean and so does not satisfy the internal axiom
of choice. Finally we observe that in homotopical algebra it is common
to consider two other definitions of Kan fibration in simplicial
presheaves: projective and injective. The wfs of projective Kan
fibrations is cofibrantly generated by a set of maps with
representable codomain and so definable. For injective Kan fibrations
the situation in general in unclear, but if \(\mathcal{D}\) is an
inverse category, as for simplicial presheaves on the walking arrow,
then we can apply the construction of the universe in \cite[Section
12]{shulmanuidhc} to see that the wfs is definable, while the more
sophisticated techniques of \cite{shulmaninftytopunivalence}
allow one to replace
the wfs of injective Kan fibrations with a different, non-full, notion
of structure, which is definable and has the same underlying class of
maps for any small category \(\mathcal{D}\).

\subsection{Hurewicz fibrations in topological spaces and related
  examples}
\label{sec:hurew-fibr-topol}

\begin{thm}
  The awfs of Hurewicz fibrations is not definable in any
  of the following categories:
  \begin{enumerate}
  \item Topological spaces
  \item The function realizability topos
  \item The Kleene--Vesley topos
  \end{enumerate}
\end{thm}

\begin{proof}
  We first consider topological spaces. We note that in the
  commutative cube below the top and bottom faces are pushouts and all
  side faces are pullbacks.
  \begin{displaymath}
    \begin{tikzcd}[sep=1em]
      & (-\infty, 1) \times \RR \ar[dd] \ar[rr] & & \RR \times \RR \ar[dd] \\
      (-1, 1) \times \RR \ar[ur] \ar[rr, crossing over] \ar[dd] & & (-1, \infty)
      \times \RR \ar[ur] & \\
      & (-\infty, 1) \ar[rr] & & \RR \\
      (-1, 1) \ar[rr] \ar[ur] & & (-1, \infty)  \ar[ur] \ar[from=uu,
      crossing over] &
    \end{tikzcd}
  \end{displaymath}

  We will show there are multiple Hurewicz fibration structures that
  all agree on the two pushout inclusions.

  Suppose we are given a lifting problem as below:
  \begin{displaymath}
    \begin{tikzcd}
      Z \times 1 \ar[d, "Z \times \delta_0"] \ar[r, "h"] & \RR \times
      \RR \ar[d] \\
      Z \times \II \ar[r, "k"] & \RR
    \end{tikzcd}
  \end{displaymath}

  For each $c \in \RR$, we define a diagonal filler
  \(j_c : Z \times \II \to \RR \times \RR\) by the following formula:
  \begin{equation*}
    j_c(z, x) := (h(z) + c \min(k(z, 0) + 1, 0) \max(k(z, x) - 1, 0) ,
    k(z, x))
  \end{equation*}
  Note that if the homotopy \(k\) factors through the inclusion
  \((-\infty, 1) \hookrightarrow \RR\) then \(\max(k(z, x) - 1, 0) =
  0\) for all \(z, x\), and if it factors through the inclusion
  \((-1, \infty) \hookrightarrow \RR\) then \(\min(k(z, 0) + 1, 0) =
  0\). In either case we have
  \begin{equation*}
    j_c(z, x) := (h(z),  k(z, x)).
  \end{equation*}
  However, it is easy to come up with examples of lifting problems
  where \(j_c\) is different for different values of \(c\). For
  example, this is the case whenever \(Z = 1\) and \(k\) is defined by
  \(k(x) := 4 x - 2\).

  For the function realizability and Kleene--Vesley topos, we simply
  use the embedding of countably based \(T_0\) spaces into the
  function realizability topos \cite{bauereqsptte}, and observe that
  \(j_c\) is computable whenever \(h\) and \(k\) are.
\end{proof}

\subsection{Non definability of Kan fibrations from logical properties
  of the interval}

We now give two classes of examples of non definable awfs's. In both
cases we use the internal logic of a topos to construct similar
examples to the one in Section \ref{sec:hurew-fibr-topol} from certain
logical principles. The first of these is that the interval admits a
linear ordering, and the second a principle that we denote ``detachable
diagonal.'' In both cases we will construct the Kan fibration
structures in the internal logic of the topos, following Orton and
Pitts \cite{pittsortoncubtopos}.

\subsubsection{Linear intervals and simplicial sets}
\label{sec:non-locally-repr-1}

It has already been shown by Sattler that the Kan fibrations of
Example \ref{ex:gzb3} are not definable in the category of simplicial
sets, with a proof appearing in \cite[Appendix
D]{vdbergfaber}. In this section we will see a new, more general proof
of this fact.

Recall that the category of simplicial sets is the classifying topos
for linear intervals with disjoint endpoints. In particular $\Delta_1$ is the
generic such, with order relation given by degeneracy maps
$\Delta_2 \rightarrowtail \Delta_1 \times \Delta_1$, and endpoints the
face maps $\delta_0, \delta_1 : \Delta_0 \rightrightarrows
\Delta_1$. We will show that in fact linearity of the interval
suffices to show non definability.

In the below, let \(\CC\) be a topos and \((\II, \leq, 0, 1)\) be a
linear order with endpoints in \(\CC\). Assume further that
\((\II, \leq, 0, 1)\) is non trivial in that the endpoint map
\(2 \to \II\) is not a regular epimorphism. Equivalently, the
following statement does \emph{not} hold in the internal logic of
\(\CC\):
\begin{displaymath}
  \forall x \in \II\;x = 0 \,\vee\, x = 1
\end{displaymath}
Note that any connected interval with disjoint endpoints is non
trivial in this sense.

From linearity, we can show that $\II \times \II$ is the union of the
two subobjects defined by $T_0 := \{(x, y) \;|\; x \geq y \}$ and
$T_1 := \{(x, y) \;|\; x \leq y \}$. We define in the internal
language a family of objects indexed over $\II \times \II$ by
$Z_{x, y} := \{ \varphi \in \Omega \;|\; x \geq y \,\rightarrow\,
\varphi \}$. Note that the pullback of $Z$ along each of the
inclusions $T_i \hookrightarrow \II \times \II$ is a trivial
fibration, in the strongest sense, that we have a choice of lift
against all monomorphisms.

\begin{lemma}
  \label{lem:sssetnodefmainlemma}
  We construct two different fibration structures, in the sense of
  Example \ref{ex:gzb3} on $Z \to \II \times \II$ that are equal when
  restricted to $T_0$ and when restricted to $T_1$.
\end{lemma}

\begin{proof}
  We work in the internal logic of \(\CC\). We will just
  define fillers for paths in the direction $0$ to $1$, the
  other direction being similar.

  Suppose we are given $\psi \in \Omega$, a path
  $p : \II \to \II \times \II$, an element $z_0$ of
  $Z_{p(0)}$ and a dependent function
  $f : \prod_{x : \II} \psi \to Z_{p(x)}$ such that
  $\prod_{w : \psi} f(0, w) = z_0$.

  For the first fibration structure, we define
  $q : \prod_{x : \II} Z_{p(x)}$, as follows.
  \begin{multline*}
    q(x) := \sum_{w : \psi} f(x, w) \;\vee\; (x = 0 \wedge z_0)
    \;\vee\; p(x) \in T_0 \;\vee \\
    p(0) = (0, 0) \,\wedge\, \pi_1(p(x)) = 1 \,\wedge\,
    \prod_{w : \psi} f(x, w)
  \end{multline*}
  Note that the clause $p(x) \in T_0$ ensures that $q(x)$ belongs to
  $Z_{p(z)}$. We also need to check the boundary conditions. We
  clearly have $z_0 \rightarrow q(0)$. It remains to check $q(0)
  \rightarrow z_0$. To do this we show that each clause in the
  disjunction defining $q(0)$ implies $z_0$. For $\sum_{w : \psi} f(x,
  w)$ we apply the assumption that $\prod_{w : \psi} f(0, w) =
  z_0$. The second clause $x = 0 \wedge z_0$ is clear. For the third
  clause $p(0) \in T_0$, note that this implies $z_0 = \top$. For the
  final clause we note that $p(0) = (0, 0)$ and $\pi_1(p(0)) = 1$
  gives a contradiction, making the final clause equal to $\bot$. We
  can similarly show the boundary condition for the partial elements.

  For the second fibration structure, we define
  $r : \prod_{x : \II} Z_{p(x)}$, as follows.
  \begin{multline*}
    r(x) := \sum_{w : \psi} f(x, w) \;\vee\; (x = 0 \wedge z_0)
    \;\vee\; p(x) \in T_0 \;\vee \\
    \pi_1(p(0)) = 0 \,\wedge\, \pi_1(p(x))
    = 1 \,\wedge\, \prod_{w : \psi} f(x, w)
  \end{multline*}
  A similar argument to before shows that $r$ satisfies the boundary
  conditions.

  We check that $q$ and $r$ agree whenever $p$ lies entirely in $T_0$
  and whenever it lies entirely in $T_1$. The former is trivial. For
  the latter, note that for any element of $T_1$ of the form $(x, y)$
  we have by definition $x \leq y$ and so $y = 0$ if and only if $(x,
  y) = (0, 0)$, so we can see that $q(x)$ and $r(x)$ are equivalent
  for all $x$.
  
  We will show these give different values for a lifting problem
  against the pushout product of
  $2 \hookrightarrow \II$ and $\delta_0 : 1 \to
  \II$. In the internal logic we view this as a family of paths
  indexed by $\II$, say a path $p_x : \II \to \II
  \times \II$ for each $x \in \II$, taken together with
  $\psi_x := (x = 0) \vee (x = 1)$ and partial elements that we need
  to define. We define $p_x(y) := (x, y)$ and take the partial elements
  $f$ to be constantly equal to $\top$.

  We can then compute $q_x(1) = x = 0 \vee x = 1$ and $r_x(1) =
  \top$. However, composing with the inclusion
  $Z_{(x, 1)} \hookrightarrow \Omega$ these give two different
  subobjects of $\II$ by the non triviality condition on \(\II\).
\end{proof}

There are several possible ways to define Kan fibrations in simplicial
sets that are known to give the same class of maps when working in a
classical setting (see e.g. \cite[Chapter IV, Section
2]{gabrielzisman}). By Corollary \ref{cor:setreprcodtolocrepr} we know
that the awfs generated by the set indexed family of horn inclusions
is definable. Also, as remarked by Shulman, the full notion of
structure for the underlying wfs is definable, assuming the axiom of
choice (which is strictly necessary by Theorem
\ref{thm:wfsnotlocrepr}). However, none of the other commonly
considered awfs's on simplicial sets are definable.

\begin{cor}
  The awfs's cofibrantly generated by the following classes of maps
  are not definable in simplicial sets:
  \begin{enumerate}
  \item \label{it:ppmono} Pushout product of mono and endpoint
    inclusion. (Example \ref{ex:gzb3})
  \item \label{it:pplocdec} Pushout product of locally decidable mono
    and endpoint inclusion.
  \item \label{it:ppsubrepr} Pushout product of the set of subobjects
    of representables and endpoint inclusion.
  \item \label{it:ppbdyinc} Pushout product of boundary inclusions and
    endpoint inclusion. (Example \ref{ex:gzb2})
  \item \label{it:enrichedlp} Fibred lifting problem against the
    coproduct of all horn inclusions with respect to the codomain
    fibration. This is same as enriched lifting problem for cartesian
    monoidal product.
\end{enumerate}
\end{cor}

\begin{proof}
  We directly considered \ref{it:ppmono} in Lemma
  \ref{lem:sssetnodefmainlemma}. Note that in each awfs, each
  generating left map is a left map of \ref{it:ppmono}, and so we do
  get two right map structures for each of the awfs's. To show that
  the two right map structures are different in Lemma
  \ref{lem:sssetnodefmainlemma} we considered a lifting problem
  against the pushout product of \(2 \hookrightarrow \II\) and
  \(\delta_0 : 1 \to \II\). This is a left map in \ref{it:pplocdec},
  \ref{it:ppsubrepr} and \ref{it:ppbdyinc}, so exactly the same proof
  applies for each of these. This only leaves \ref{it:enrichedlp}, for
  which we observe that a similar argument applies to
  \(\II \times \delta_0\).
\end{proof}

We also note that the same argument applies to Hurewicz fibrations:
\begin{cor}
  Let \(\CC\) be a topos and \((\II, \leq, 0, 1)\) be a non trivial
  linear order with endpoints in \(\CC\). Then Hurewicz fibrations are
  not definable.
\end{cor}

\subsubsection{Detachable diagonal and BCH cubical sets}
\label{sec:non-locally-repr}

We recall that BCH cubical sets \cite{bchcubicalsets} possess a non-cartesian monoidal
product, \emph{separated product}. Bezem, Coquand and Huber defined
Kan fibrations to be maps with the right lifting property against the
category indexed family given by pushout product of a maps
$0 \to \square_n$ a boundary inclusion
$\partial \square_m \hookrightarrow \square_m$ and an endpoint
inclusion (the last two can also be merged together to give an ``open
box inclusion''). A morphism is a pair of maps
$\square_n \to \square_{n'}$ and $\square_m \to \square_{m'}$, which
induces a morphism the corresponding maps in the pushout product. It
is clear from the construction of the universe by Bezem, Coquand and
Huber that this gives a definable awfs, although we can
now also see definability as an instance of Corollary
\ref{cor:catindexedtinytolocrepr}.

One might wonder what happens if instead of separated product we use
cartesian product in the pushout product in the generating family of
trivial cofibrations. We will give a concrete reason that this is a
bad idea: the resulting awfs is not definable. Just as for simplicial
sets, we will identify a logical property of the interval that
suffices to carry out the argument.

We first note that $\II \times \II$ can be constructed by ``pasting a
diagonal to $\II \otimes \II$.'' More formally we have the following
lemma.

\begin{lemma}
  \label{lem:sqplusdiag}
  The following diagram is both a pushout and a pullback.
  \begin{equation*}
    \begin{tikzcd}[sep = 3.8em]
      1 + 1 \ar[r, "{[\langle 0, 0 \rangle, \langle 1, 1 \rangle]}"]
      \ar[d, "{[0, 1]}"] &
      \II \otimes \II \ar[d] \\
      \II \ar[r, "\Delta"] & \II \times \II
    \end{tikzcd}
  \end{equation*}
\end{lemma}

\begin{proof}
  E.g. this can be seen clearly using Pitts' presentation of the
  category as $01$-substitution sets \cite{pittsnompcs}.
\end{proof}

We can understand this pushout in the internal logic as follows. The
lemma tells us directly that \(\II \times \II\) can be written as the
union of the subobjects \(\II \otimes \II\) and
\(\Delta : \II \hookrightarrow \II \times \II\) and that the
intersection of these two subobjects is the inclusion of diagonal
endpoints \(2 \hookrightarrow \II \times \II\). In general, separated
product \(\otimes\) is not well behaved with respect to the internal
logic of cubical sets\footnote{The only closed semi cartesian monoidal
  product fibred over a codomain fibration is cartesian product: any
  such monoidal product has a fibred right adjoint by closedness and
  so by Lemma \ref{lem:particalrightadjfib} preserves opcartesian
  maps, and so we calculate
  \(A \otimes_1 B \cong A \otimes_1 \sum_B 1 \cong \sum_B (B^\ast(A)
  \otimes_B 1) \cong A \times B\).}, but in this case we can deduce
from the above statement and purely formal reasoning in the Heyting
algebra of subobjects that \(\II \otimes \II\) can be defined from the
diagonal inclusion via Heyting implication. Hence BCH cubical sets
satisfy the following:

\begin{defn}
  Let \(\BB\) be a topos with interval object \(\II\). We say \(\II\)
  has \emph{detachable diagonal} if the following statement holds in
  the internal language of \(\BB\).
  \begin{displaymath}
    \forall i, j \in \II, \qquad i = j \quad\vee\quad (i = j \;\to\; i = 0 \vee
    i = 1) 
  \end{displaymath}
\end{defn}

\begin{rmk}
  Once again our main example of a topos with this property is in fact the
  classifying topos. To make this precise, note that given an interval
  \(\II\) with disjoint endpoints and detachable diagonal, we can
  define a binary relation \(- \# -\) by taking \(x \# y\) when \(x =
  y \to x = 0 \vee x = 1\). This then defines a model for the
  following geometric theory:
  \begin{align*}
    x \# y &\vdash y \# x & x \# x &\vdash x = 0 \vee x = 1 & &\vdash x = y \vee x \# y \\
           & \vdash x \# 0 & & \vdash x \# 1 & 0 = 1 & \vdash \bot
  \end{align*}
  BCH cubical sets are the classifying topos for this theory with
  generic object \(\II\), where the binary relation \(-\#-\) is the
  canonical map \(\II \otimes \II \rightarrowtail \II \times \II\).
\end{rmk}

\begin{thm}
  Suppose that we are given a topos with connected interval object
  with disjoint endpoints and detachable diagonal.

  Then the awfs cofibrantly generated by pushout product of
  monomorphisms and endpoint inclusions (with respect to cartesian
  product) is not definable.
\end{thm}

\begin{proof}
  We define for each $x, y \in \II$ a set $P_{x, y}$ as follows. We
  first define
  \begin{displaymath}
    Q_{x, y} := \{ 0, 1 \in 2 \;|\; x = y \} + \{ 2, 3 \;|\; x = y \to x = 0
    \vee x = 1 \}
  \end{displaymath}
  We define an equivalence relation $\sim$ on $Q$ by setting $0 \sim
  2$ and $1 \sim 3$ when $x = y = 0$ and setting $0 \sim 3$ and $1
  \sim 2$ when $x = y = 1$. We define $P_{x, y}$ to be the quotient
  $Q_{x, y} / {\sim}$.

  Note that by assumption we can write $\II \times \II$ as a union of
  two subobjects: the diagonal
  $\{(x, y) \in \II \times \II \;|\; x = y \}$ and the subobject
  $C := \{(x, y) \in \II \times \II \;|\; x = y \to x = 0 \vee x =
  1\}$. We will show that the restriction of $P$ to either subobject
  is isomorphic to the family constantly equal to $2$, and so a Kan
  fibration, whereas \(P\) itself is not.

  First consider the diagonal. In this case each $P_{x, x}$ contains
  equivalence classes $[0]$ and $[1]$. It can only contain the
  equivalence class $[2]$ when $x = 0 \vee x = 1$. However, in the
  former case $[2] = [0]$ and in the latter case $[2] =
  [1]$. Similarly, it can only contain an equivalence class $[3]$ when
  it is equal to either $[0]$ or to $[1]$. Hence $P_{x, x} \cong 2$.

  Now consider the case where $(x, y) \in C$. In this case $P_{x, y}$
  definitely contains the equivalence relations $[2]$ and
  $[3]$. However, it can only contain $[0]$ when it is identified with
  either $[2]$ or with $[3]$, and similarly for $[1]$. Hence we have
  $P_{x, y} \cong \{2, 3\} \cong 2$.

  We now define a family of paths $p_x$ in $\II \times \II$ by setting
  $p_x(y) := (x, y)$. We define $z \in P_{p(0)}$ to be $[2]$. If $P$ is
  a Kan fibration, then we would have a family of fillers $j_x :
  \prod_{y : \II}P_{x, y}$. Note that by the explicit description of
  $P_{x, x}$ above, we have for all $x$ that $j_x(x) = [0]$ or $j_x(x)
  = [1]$. Since $j_0(0) = z_0 = [2] = [0]$, and using the
  connectedness of the interval, we have $j_1(1) = [0]$.

  Now using the explicit description of $P_{x, y}$ for $(x, y) \in C$,
  we see that each $j_1(x)$ must be either equal to $[2]$ or to
  $[3]$. Again using $z_1$ and the connectedness of the interval, we
  have $j_1(1) = [2]$. However, $[0]$ and $[2]$ are not equal as
  elements of $P_{1, 1}$, giving a contradiction.  
\end{proof}

\begin{rmk}
  Since we showed there is no fibration structure at all on the
  pushout, we can show that for presheaf categories where the interval
  has detachable diagonal the ``canonical'' universe, constructed in
  Theorem \ref{thm:locrepinpshbase} is not fibrant. If it was, we
  would be able to construct a map from \(\II\) into the universe
  using the universal property of the pushout, and thereby pull back
  the fibration structure to the pushout.
\end{rmk}

\section{Conclusion}
\label{sec:conclusion}

Definability is a fundamental notion in the theory of Grothendieck
fibrations that characterises when external properties and structure
can be accessed from within the internal logic of the base of a
fibration. It has appeared in many different guises over time. In this
paper we gave a comprehensive overview uniting the theory of
definability developed by Lawvere, B\'{e}nabou and Johnstone with the
separate thread starting with Cisinski's definition of local fibration
\cite[Definition 3.7]{cisinskiuniv}, further developed by Sattler
\cite{sattlermodelstructures} and ending with Shulman's local
representability.

Algebraic weak factorisation systems can be viewed as monadic notions
of structure equipped with a composition functor. As notions of
structure they lie on the boarder between definability and non
definability. On the side of definability we saw a general sufficient
criterion that encompasses some very different looking examples of
definable awfs's. By applying our result to a codomain fibration, we
recovered the definability of Kan fibrations in cubical sets \cite{lops, awodey19}. By
applying to set indexed family fibrations, we obtained a different
looking criterion, where the exponential functor used in the internal
definition of tininess is replaced with a hom set functor. The theorem
is phrased as a general condition on awfs's cofibrantly generated by a
family of maps in a fibration, that includes Kan fibrations generated
by a tiny interval, but also other examples. In particular in Example
\ref{ex:natmodelex} we saw an example of awfs's in natural models that
made essential use of a tiny family of objects that is not simply
generated by one tiny object. The general result includes most
examples of cofibrantly generated awfs's used in the semantics of
homotopy type theory. However, we leave two interesting classes of
examples as a direction for future work. The first is awfs's
cofibrantly generated by a double category, such as the definition of
Kan fibration due to Van den Berg and Faber in \cite{vdbergfaber}, who
gave a direct proof of definability. The
second is examples where the role of exponential in Kan fibration is
replaced by monoidal exponential, as in the definition of Kan
fibration by Bezem, Coquand and Huber in \cite{bchcubicalsets}. A
promising approach is suggested by Nuyts and Devriese in
\cite{nuytsdevriese}, who showed that the relevant right
adjoint to monoidal exponentiation is an instance of a general
construction of transpension types in presheaf categories.

Our examples of non definable awfs's included identifying logical
principles satisfied by the interval that can be used to show the non
definability of Kan fibrations. In both cases the main examples of
simplicial sets, and BCH cubical sets respectively turned out to be
classifying toposes for the structures that we considered. Simplicial
sets have long been regarded as a very natural setting for studying
the structure of topological spaces up to homotopy
\cite{gabrielzisman}, and they are used in the original model of
homotopy type theory \cite{voevodskykapulkinlumsdainess}
so it is natural to ask if the same can be done
when working constructively. We have seen that many reasonable
definitions of Kan fibration in simplicial sets are non definable.
However, we leave it as an open problem to
either show that one of the non definable versions of Kan fibration can still
be used to model univalent type theory in simplicial sets, or to show
it causes an unavoidable obstruction, in which case it is necessary to
use one of the definable versions of Kan fibration, as in \cite{gambinosattlerszumilo} or
\cite{vdbergfaber}, or to avoid simplicial sets entirely in favour of
other categories such as cubical sets. In BCH cubical sets we saw a
more severe example of the kind of thing that can happen in the
absence of definability. BCH cubical sets are a presheaf topos, and so
very well behaved as a category, and possesses an obvious choice of
interval object. As such, one might na\"{i}vely expect that as an
alternative to the original monoidal definition of Kan fibration, it
would be possible to construct a model of homotopy type theory using
the cartesian definition of Kan fibration. However, the only apparent
choice of universe classifying Kan fibrations fails to be a Kan
fibration itself. We can see from this that definability is a key
property to consider when constructing models of homotopy type
theory.

\bibliographystyle{alpha}
\bibliography{mybib}{}

\end{document}